\documentclass[11pt, reqno, a4paper]{amsart}
\usepackage{amssymb,epsfig,amsmath,amsthm}
\usepackage{amsmath,amssymb,graphicx}
\usepackage{color}
\graphicspath{{Graphix/}}

\usepackage{pdfsync}

\usepackage{environments}
\usepackage{fullpage}

\numberwithin{equation}{section}
\numberwithin{figure}{section}

\newcommand{\dx}{\,\mathrm{d}x}
\newcommand{\dX}{\,\mathrm{d}X}

\newcommand{\dz}{\,\mathrm{d}z}
\newcommand{\dk}{\,\mathrm{d}k}

\newcommand{\sign}{\mathrm{sign}\onept}

\newcommand{\R}{\mathbb{R}}

\newcommand{\Z}{\mathbb{Z}}

\newcommand{\bu}{\boldsymbol{u}}

\newcommand{\bs}[1]{\boldsymbol{#1}}
\newcommand{\e}{\mathrm{e}}
\newcommand{\ii}{\mathrm{i}}
\newcommand{\mc}[1]{\mathcal{#1}}
\newcommand{\coloneqq}{\mathrel{\mathop:}=}
\newcommand{\eqqcolon}{=\mathrel{\mathop:}}
\newcommand{\intO}{\int_\Omega}

\newcommand{\norm}[1]{\lVert #1 \rVert}
\newcommand{\snorm}[1]{\lvert #1 \rvert}
\newcommand{\tworow}[2]{\genfrac{}{}{0pt}{}{#1}{#2}}
\newcommand{\mrm}[1]{{#1}}

\renewcommand{\tilde}{\widetilde}
\newcommand{\smallhat}{\hat}
\renewcommand{\hat}{\widehat}

\newcommand{\Didot}{\!\cdot\!}

\newcommand{\Hper}{{H}^1_{\#}(\Omega)}

\newcommand{\Hne}{{H}_0^1(\Omega)}

\newcommand{\HneN}[1]{\left\lVert\nabla #1\right\rVert_{{L}^2}}

\newcommand{\veps}{\varepsilon}

\newcommand{\supp}{\mathrm{supp}\hspace{1pt}}

\newcommand{\Sm}{\mathrm{S}}

\newcommand{\Ea}{\E^{\mathrm{at}}}
\newcommand{\Ec}{\E^{\mathrm{cb}}}
\newcommand{\Eqc}{\E^{\mathrm{ac}}}
\newcommand{\bya}{\by_{\mathrm{at}}}

\newcommand{\mrma}{\mathrm{at}}
\newcommand{\mrmc}{\mathrm{cb}}
\newcommand{\aL}{a_L}

\newcommand{\aR}{a_R}
\newcommand{\aLR}{a}

\newcommand{\E}{\mathcal{E}}

\newcommand{\by}{{\boldsymbol{y}}}

\newcommand{\U}{\mathcal{U}}
\newcommand{\Deltaa}{\Delta \aLR}
\newcommand{\ceff}{\mu}
\newcommand{\onept}{\hspace{1pt}}

\renewcommand{~}{\hspace{1.2mm}}

\newcommand{\sigmaa}{\sigma}
\newcommand{\vsig}{\varsigma_0}

\def\calS{\mathcal{S}}
\def\cb{{\rm cb}}
\def\qc{{\rm ac}}
\newcommand{\smfrac}[2]{{\textstyle \frac{#1}{#2}}}

\usepackage{hyperref}

\begin{document}

\title{Atomistic-to-Continuum Coupling Approximation of a One-Dimensional Toy Model
  for Density Functional Theory}

\author{B. Langwallner}
\address{B. Langwallner\\ Mathematical Institute\\
  24-29 St Giles' \\ Oxford OX1 3LB \\ UK}
\email{langwallner@maths.ox.ac.uk}

\author{C. Ortner}
\address{C. Ortner\\ Mathematics Institute \\ Zeeman Building \\
  University of Warwick \\ Coventry CV4 7AL \\ UK}
\email{christoph.ortner@warwick.ac.uk}

\author{E. S\"{u}li}
\address{E. S\"{u}li\\ Mathematical Institute\\
  24-29 St Giles' \\ Oxford OX1 3LB \\ UK}
\email{suli@maths.ox.ac.uk}

\date{\today}

\thanks{This work was supported by the EPSRC Critical Mass Programme
  ``New Frontiers in the Mathematics of Solids'' (OxMoS) and by the
  EPSRC Grant ``Analysis of Atomistic-to-Continuum Coupling
  Methods''.}



\begin{abstract}
  We consider an atomistic model defined through an interaction field
  satisfying a variational principle, and can therefore be considered
  a toy model of (orbital free) density functional theory.  We
  investigate atomistic-to-continuum coupling mechanisms for this
  atomistic model, paying special attention to the dependence of the
  atomistic subproblem on the atomistic region boundary and the
  boundary conditions. We rigorously prove first-order error estimates
  for two related coupling mechanisms.
\end{abstract}

\maketitle

\section{Introduction}
\label{sec:SM_Intro}
The quasicontinuum (QC) method and, more generally,
atomistic/continuum coupling (a/c) methods, are numerical
coarse-graining techniques for the efficient simulation of phenomena
and processes in materials at the nano-scale, such as defects,
fracture, grain boundaries, or nano-indentation \cite{Tadmor1996,
  Tadmor1996mixed, Shenoy1998quasicontinuum,
  Miller1998quasicontinuum}. Incompatibilities between the treatment
of forces in atomistic and continuum models lead to difficulties in
defining coupling mechanisms that do not introduce additional
errors. Substantial effort has been made to understand this problem
and to construct efficient and accurate a/c methods; see
\cite{Shimokawa, ELuYang_Uniform, Shapeev, Gavini_Field,
  xiao2004bridging} for examples of formulations of computational
methods and \cite{BlancLeBrisLegoll, DL_ghostforceoscillation,
  DLO_Accuracy, MingYang_QNL, Ortner2011:patch, OrtnerQNL} and
references therein for examples of analytical treatments. Formulations
of a/c methods for atomistic models based on quantum mechanics were
proposed in \cite{GaBhOr07, GarcCervLuE}, but, to the best of our
knowledge, no rigorous analysis of these methods exists.


In the present article we formulate and analyze one-dimensional a/c
methods for an atomistic model that is defined through an interaction
field satisfying a linear variational principle. Our results are
related to two classes of a/c methods: Firstly, our work can be viewed
as an analysis of (a simplified version of) the a/c method
proposed by Iyer and Gavini \cite{Gavini_Field}, who use field-based
versions of classical potentials to formulate their method. Secondly,
the atomistic model we formulate can be considered a toy model of
(orbital free) density functional theory, and hence our work
represents a preliminary step towards a rigorous analysis of the
a/c methods described in \cite{GaBhOr07, GarcCervLuE}.

The article is structured as follows. In Section \ref{sec:SM_Intro} we
formally motivate the atomistic model, and introduce the necessary
notation. In Section \ref{SM:ModelPeriodic} we give a precise
formulation of the model with periodic boundary conditions and derive
a ``weak formulation'' for the resulting forces on the
particles. Section \ref{sec:SMDirBCs} is devoted to the analysis of
the model in a bounded domain when the fields are subjected to
Dirichlet boundary conditions. The Cauchy--Born continuum model is
derived and analyzed in Section \ref{sec:SM_CB}. Finally, in Sections
\ref{Sec:QCCoupling} and \ref{sec:method_dir} we propose two possibile
constructions of a/c methods based on different exchange of boundary
conditions between an atomistic and continuum region, and establish
error estimates.

\subsection{Field-based formulation of pair interactions}
\label{sec:SM_Atomistic}
The following outline follows ideas presented in \cite{Gavini_Field}.
Let $\by=(y_1,\ldots,y_N)\in\R^{N}$ represent the coordinates of $N$
particles in one dimension. We consider an atomistic energy based on a
pair-potential $V$,
\begin{equation*}
 \E(\by) = \frac{1}{2}\sum_{\tworow{i,j=1}{i\neq j}}^N V(|y_i-y_j|).
\end{equation*}
The force on particle $i$ is given by
\begin{equation*}
-D_{y_i} \E(\by) = -\sum_{\tworow{j=1}{j\neq i}}^N \sign(y_i-y_j) V'(|y_i-y_j|).
\end{equation*}
We note that the forces are nonlocal expressions in the sense that
their computation involves summation over the other $N-1$ particles.

Next, we make a few modifications to this model. First, we replace the
pointwise particles with smooth, nonnegative, and compactly supported
particle densities $\delta_\veps(\cdot-y_i)$ (such that
$\int_\R\delta_\veps(x)\dx = 1$). This leads to
\begin{equation*}
 \mathcal{E}(\by)  \approx \frac{1}{2}\sum_{\tworow{i,j=1}{i\neq j}}^N\int_\R \int_\R
\delta_\veps(z-y_i) V(|z-x|) \delta_\veps(x-y_j)\dz\dx.
\end{equation*}
To simplify the presentation further, we include the self-energies of
the individual particle densities and define
\begin{equation*}
 \mathcal{E}_\veps(\by)  = \frac{1}{2}\sum_{i,j=1}^N\int_\R \int_\R
\delta_\veps(z-y_i) V(|z-x|) \delta_\veps(x-y_j)\dz\dx.
\end{equation*}
This additional self-energy contribution does not affect the
forces. It can be computed explicitly and subtracted from the energy
later on. Upon introducing the field $\phi:\R\rightarrow\R$,
\begin{equation}
 \phi(x) = \int_\R \rho_\by(z)V(|x-z|)\dz, \quad\text{where}\ \ \ \rho_\by(z) =
\sum_{i=1}^N\delta_\veps(z-y_i),
\label{eq:phiconvolution}
\end{equation}
to rewrite the energy $\E_\veps(\by)$ in the form
\begin{equation*}
 \E_\veps(\by) = \frac{1}{2}\int_\R \rho_\by(x)\phi(x)\dx.
\end{equation*}
It is now easy to see that the forces are given by the \emph{local}
expression
\begin{equation*}
 -D_\by\E_\veps(\by) = -\int_\R D_\by\rho_\by(z)\phi(z)\dz.
\end{equation*}
Hence, if the field $\phi$ is known, then it becomes unnecessary to
compute nonlocal sums over particles. The nonlocality of the
interaction has been encoded in the field $\phi$. However, it is now
necessary to compute the field $\phi$, which is defined via the
convolution \eqref{eq:phiconvolution}.

Suppose that the pair-potential $V$ is the Green's function for a
linear differential operator $L_V(\nabla)$; then, $\phi$ can
alternatively be computed by solving the differential equation
\begin{equation*}
 L_V(\nabla) \phi =\rho_\by.
\end{equation*}

As an example we consider the Yukawa potential in one space dimension
\begin{equation*}
V(x) = \frac{1}{2m}\e^{-m|x|}  = \frac{1}{2\pi}\int_\R \frac{1}{k^2+m^2}\,\e^{\ii kx}\dk.
\end{equation*}
In this case $\phi$ can be obtained as the solution to 
\begin{equation*}
 -\Delta\phi + m^2\phi =\rho_\by
\end{equation*}
or, equivalently, as a solution to the minimization problem
\begin{equation*}
 \phi = \arg\min_{\varphi}\biggl\{\frac{1}{2}\int_\R
|\nabla\varphi|^2+m^2\varphi^2\dx - \int_\R \rho_{\by}\varphi\dx\biggr\}.
\end{equation*}
The resulting interaction potential $\E_\veps$ can also be written in
the form
\begin{equation}
\E_\veps(\by) =- \min_{\varphi}\biggl\{\frac{1}{2}\int_\R
|\nabla\varphi|^2+m^2\varphi^2\dx - \int_\R \rho_{\by}\varphi\dx\biggr\}.
\label{eq:Eintrodef}
\end{equation}

The present work is devoted to the analysis of a/c approximations of
\eqref{eq:Eintrodef} in a periodic one-dimensional setting. What
distinguishes this analysis from previous analyses of a/c methods is
that the coupling is achieved through an exchange of boundary
conditions for the interaction field $\phi$, rather than ghost-force
removal ideas such as \cite{ELuYang_Uniform, Shapeev}.

\begin{remark}
  The interaction defined by \eqref{eq:Eintrodef} is purely
  repulsive. A purely attractive interaction can be obtained by
  changing the outer minus sign in the definition of $\E_\veps$ to a
  plus sign. We could combine two energies of the form
  \eqref{eq:Eintrodef} with different parameters $m$ to model an
  interaction similar to the Morse potential $V(|x|) = \e^{-2|x|} -
  2\e^{-|x|}$ \cite{Gavini_Field}.
\end{remark}

\subsection{Notation}
\label{sec:SM_Notation}
We consider an infinite chain of atoms on the one-dimensional lattice
$\smallhat{\bs{X}}=\veps\Z$, where $\veps = 2/(2N+1)$ is the reference
lattice spacing. Moreover, to keep the analysis simple, we admit only
$(2N+1)$-periodic displacements from the reference lattice
(cf. \cite{OrtnerQNL}). Hence, we define the spaces of admissible
displacements and deformations, respectively, by
\begin{align*}
  \mc{U} =~& \big\{\bu\in\R^{\Z}: u_{j+(2N+1)}=u_{j}\ \ \forall j\in\Z,\ \
  {\textstyle \sum_{j=-N}^N} u_j =
  0\big\}, \quad \text{and} \\
 \mc{Y} =~& F\hat{\bs{X}} + \mc{U},
\end{align*}
where $F > 0$ is a prescribed {\em macroscopic strain}. A deformation
$\by \in \mc{Y}$ defines the
computational domain 
\begin{equation*}
  \Omega = (y_{-N-1},y_N)
\end{equation*}
for the field variable $\phi$. We note that the length of the interval
is independent of $\by$.

We define the finite differences $\by',\by'' \in\mc{U}$ for $\by\in\mc{Y}$ or $\mc{U}$ by
their respective components
\begin{equation*}
 y_j' = \frac{y_j-y_{j-1}}{\veps},\qquad y_j'' = \frac{y_{j+1}-2y_j+y_{j-1}}{\veps^2}.
\end{equation*}
Let us also define the weighted $\ell^2$ scalar product and norm by
\begin{equation}
 (\bu,\bs{v})_\veps = \veps  \sum_{\nu = -N}^N
u_\nu v_\nu\quad \forall \bu,\bs{v}\in\mc{U}, \qquad \norm{\bu}_{\ell^2_\veps} \coloneqq
(\bu,\bu)_\veps^{1/2}\quad \forall \bu\in\mc{U}.
\end{equation}
The $\ell^\infty$-norm is defined in the obvious way 
\begin{equation*}
\norm{\bu}_{\ell^\infty} = \max_{\nu=-N,\ldots,N} |u_\nu|\quad \forall \bu\in\mc{U}. 
\end{equation*}
The space $\mc{U}$ equipped with the discrete Sobolev seminorm
$\lVert\bu \rVert_{\mc{U}^{1,2 }} = \lVert\bu' \rVert_{\ell^2_\veps}$
will be denoted by $\mc{U}^{1,2}$ and its topological dual space by
$\mc{U}^{-1,2}$. The norm on $\mc{U}^{-1,2}$ is given by
\begin{equation*}
 \lVert T\rVert_{\mc{U}^{-1,2}} = \sup_{\bu\in\mc{U}^{1,2}}\frac{T\bu}{\norm{\bu}_{\mc{U}^{1,2}}}.
\end{equation*}
For monotonically increasing $\bs{y}\in\mc{Y}$ (which we will write as
$\by'>0$) we denote by $\Sm(\by)\subset\mrm{H}^1(\Omega)$ the space of
continuous functions that are linear on every interval
$Q_i=(y_{i-1},y_{i})$, $i\in\{-N,\ldots,N\}$. Furthermore, we define
$\Sm_\#(\by) = \Sm(\by)\cap\mrm{H}^1_\#(\Omega)$ to be the subset of
all periodic functions in $\Sm(\by)$.

\section{Periodic Boundary Conditions}
\label{SM:ModelPeriodic}
\begin{figure}
\begin{center}
\includegraphics[width=.7\linewidth]{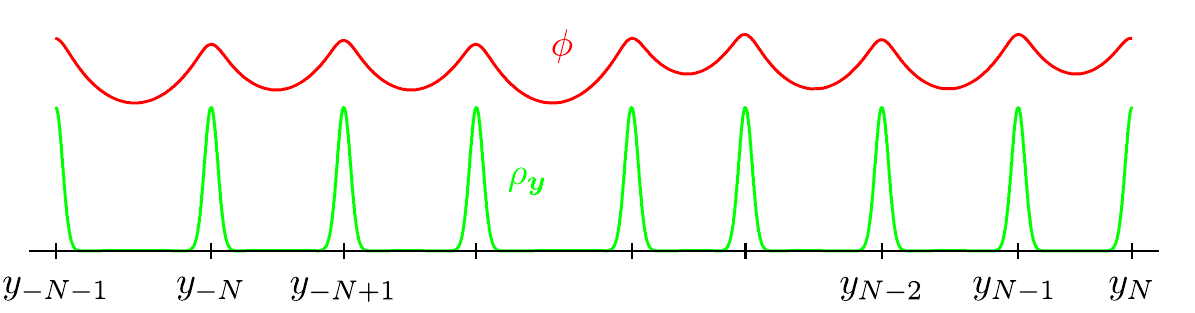}
\caption{Sketch of the basic atomistic problem: the field $\phi$ is periodic in
$\Omega=(y_{-N-1},y_N)$ and $\rho_\by$ is a smooth particle density representing the atoms with
positions given by $\by\in\mc{Y}$.}
\label{fig:QC1}
\end{center}
\end{figure}
We now put the field-based interaction potential that was outlined above in a precise 
mathematical framework. Let the functional
$I:\mrm{H}^1_{\#}(\Omega)\times\mc{Y}\rightarrow\R$ be defined by
\begin{align*}
 I(\varphi,\by) =~& \int_\Omega\Bigl(\tfrac{1}{2}\varepsilon^2|\nabla\varphi|^2 + \tfrac{1}{2}
m^2\varphi^2\Bigr)\dx - \intO \rho_{\by}\varphi\dx\!, \quad
\text{where} \\
& \rho_{\by}(x) = \veps\sum_{j\in\Z} \delta_\varepsilon(x-y_j),\quad\text{and}\quad
\delta_\varepsilon(x) =  \veps^{-1}\delta_1 (x/\varepsilon).
\end{align*}
Here, $\delta_1$ is a symmetric, nonnegative, regularized delta
distribution with compact support
$\bigl[-\tfrac{\vsig}{2},\tfrac{\vsig}{2}\bigr]$, where $\vsig>0$ and
$\int_\R\delta_1 \dx= 1$; see Figure \ref{fig:QC1}. We will frequently
refer to the paramter $\vsig$, which is fixed throughout the paper.

We then define the interaction potential $\E:\mc{Y}\rightarrow\R$ by
\begin{equation}
\E(\by) = -\min_{\varphi\in \Hper} I(\varphi,\by).
\label{eq:Eatomistic}
\end{equation}
The respective minimizer (see Figure \ref{fig:QC1})
\begin{equation*}
 \phi=\arg\min_{\varphi\in\mrm{H}^1_\#(\Omega)} I(\varphi,\by)
\end{equation*}
is the periodic solution to the Euler--Lagrange equation
\begin{equation}
 -\veps^2\Delta\phi + m^2\phi = \rho_\by\quad \text{in } \Omega.
\label{eq:phiPDE}
\end{equation}
Although $\phi$ depends on $\by$, we will usually suppress this in our
notation. It will always be clear from the context, which
configuration $\phi$ belongs to. It follows from \eqref{eq:phiPDE} and
integration by parts that
\begin{equation*}
 \E(\by) = \frac{1}{2}\intO \phi\rho_\by\dx.
\end{equation*}
To determine equilibrium configurations subject to a given external
force $\bs{f}\in\mc{U}^{-1,2}$ we need to minimize the total potential
energy $E_{\bs{f}}:\mc{Y}\rightarrow \R$ defined by
\begin{equation}
  E_{\bs{f}}(\by) = \E(\by) + (\bs{f},\by)_{\veps}. \label{eq:totalenergymin}
\end{equation}
A minimizer $\bar\by\in\mc{Y}$ of \eqref{eq:totalenergymin} satisfies
the following Euler--Lagrange equation in $\mc{U}^{-1,2}$:
\begin{equation*}
 DE_{\bs{f}}(\bar\by) = D\E(\bar\by) + \bs{f} = {\bf 0}.
\end{equation*}

In the following we analyze the derivatives of $\E$. In particular, we
obtain a ``weak'' formulation for the first derivative $D\E$ that acts
as a natural connection point for the coupling with a continuum model.

\begin{proposition}
\label{prop:periodicderivatives}
The potential $\E:\mc{Y}\rightarrow \R$ defined by
\eqref{eq:Eatomistic} is twice continuously Fr\'{e}chet
differentiable. The components of the first derivative are given by
\begin{equation}
D_{y_j} \E(\by) 
= -\veps\intO\nabla\delta_\varepsilon(x-y_j)\phi(x)\dx \label{eq:nablaV}
\end{equation}
for $j\in\{-N,\ldots,N-1\}$ and by
\begin{equation}
 D_{y_N} \E(\by) = -\veps\intO\bigl(\nabla\delta_\varepsilon(x-y_{-N-1}) +
\nabla\delta_\veps(x-y_N)\bigr)\phi(x)\dx.
\label{eq:nablaVyN}
\end{equation}
\end{proposition}
\begin{proof}
  The proof of this result is standard and can be found in
  \cite{Gavini_Field}, for example.
\end{proof}

We stress the fact that the forces $-D_\by\E(\by)$ are local
expressions. To calculate the force on atom $j$ it is necessary to
know $\phi$ in $\supp\delta_\varepsilon(\cdot-y_j)$ but there is no
need to sum over all remaining atoms. This nonlocality is encoded in
the field $\phi$.

Next we establish the \emph{weak formulation} for the forces on
particles. This very much resembles the structure of the continuum
equations and will be the basis for the a/c coupling in Section
\ref{Sec:QCCoupling}. A version of this calculation was already shown
in \cite{Gavini_Field2}, which used an interpolant for the
displacement that is constant on the support of every
$\delta_\veps(\cdot-y_j)$. To avoid this restriction, we modify and
extend the argument in \cite{Gavini_Field2}.

For simplicity we assume that the supports of the densities of
different particles do not intersect:
\begin{equation*}
\supp\delta_\veps(\cdot-y_i)\cap \supp\delta_\veps(\cdot-y_j)=\emptyset\quad \forall i,j\in\Z,\
\ i\neq j. 
\end{equation*}
Since, $|\supp \delta_\veps(\cdot-y_i)|=\veps\vsig$, this is
equivalent to $|y_{j}-y_i|> \veps\vsig$ for $i\neq j$ or, if $\by$ is
an increasing sequence, $y_j'> \vsig$ for all $j\in\Z$.

\begin{lemma}
\label{lemma:weakformulation}
Let $\by\in\mc{Y}$ satisfy $\by'>\vsig$ and let
$\phi\in\mrm{H}^1_\#(\Omega)$ be the associated field, defined by
\eqref{eq:phiPDE}. Let $\bu=(u_j)_{j\in\Z}\in \U$ be a test vector and
$u\in\Sm_\#(\by)$ the periodic piecewise linear interpolant of $\bu$,
that is, $u(y_j) = u_j$ for $j \in \Z$.
Then,
\begin{equation}
 \label{eq:per_weakform}
D\E(\by)\Didot\bu = \sum_{j=-N}^N D_{y_j}\E(\by)\Didot u_j =\ \intO \sigmaa_{\by}(x)\nabla
u(x) \dx,
\end{equation}
where $\sigmaa_{\by} = \sigmaa_{\by,1}+\sigmaa_{\by,2}$ and
\begin{equation}
\label{eq:sigmaat12}
\begin{split}
\sigmaa_{\by,1}(x) =~&
\tfrac{1}{2}\veps^2|\nabla\phi|^2-\tfrac{1}{2}m^2\phi^2+\rho_\by\phi,\\
 \sigmaa_{\by,2}(x) =~&\veps\sum_{j=-N-1}^N\phi(x)
\nabla\delta_\varepsilon(x-y_j)(x-y_j).
\end{split}
\end{equation}
\end{lemma}

\begin{proof}
  We begin by multiplying the derivative \eqref{eq:nablaV} for
  $j\in\{-N,\ldots,N-1\}$ by the component $u_j$:
\begin{align*}
D_{y_j} \E(\by) u_j =~& -\veps u_j\intO \nabla\delta_\varepsilon(x-y_j)\phi(x)\dx \\
=~& -\veps\intO u(x)\nabla\delta_\varepsilon(x-y_j)\phi(x)\dx +\veps\intO (u(x) - u_j
)\nabla\delta_\varepsilon(x-y_j)\phi(x)\dx\\
=~& \veps\intO \delta_\varepsilon(x-y_j)u(x)\nabla\phi(x)\dx +\veps\intO \delta_\veps(x-y_j)
\phi(x)\nabla u(x)\dx\\
&\ +\veps\intO (u(x) - u_j )\nabla\delta_\varepsilon(x-y_j)\phi(x)\dx\eqqcolon T_1^{(j)} +
T_2^{(j)} + T_3^{(j)}.
\end{align*}
Here we have used integration by parts but there are no boundary terms since $u$, $\phi$ and
$\rho_\by$ are periodic on $\Omega$. Using \eqref{eq:nablaVyN} we obtain a similar expression for
$D_{y_N}\E(\by)u_N$. Summing over $j=-N,\ldots,N$ we obtain
\begin{equation}
D \E(\by)\Didot \bu  = \sum_{j=-N}^N D_{y_j} \E(\by)\Didot u_j = T_1 + T_2 + T_3,\label{eq:T123}
\end{equation}
where $T_i = \sum_{j=-N}^NT_i^{(j)}$, $i\in\{1,2,3\}$. From
$\rho_\by=\veps\sum_{j\in\Z}\delta_\veps(\cdot-y_j)$ it immediately follows that
\begin{equation*}
T_2 = \intO \rho_{\by}(x) \phi(x)\nabla u(x)\dx. 
\end{equation*}
For $T_1$ we can carry out the following rearrangements
\begin{align*}
T_1 &= \intO\rho_{\by} u \nabla\phi\dx 
 = \intO\bigl(-\veps^2\Delta\phi+m^2\phi\bigr)u\nabla\phi \dx \\ 
& = \intO \bigl(-\veps^2\nabla\phi\Delta\phi+m^2\phi\nabla\phi\bigr)u \dx 
 = \frac{1}{2}\intO \nabla\bigl(-\veps^2|\nabla\phi|^2+m^2\phi^2\bigr)u \dx \\
& = \frac{1}{2}\intO \bigl(\veps^2|\nabla\phi|^2-m^2\phi^2\bigr)\nabla u \dx.
\end{align*}
Here, we have again used integration by parts and the periodicity of all functions involved. We
deduce that
\begin{equation*}
 T_1+T_2 = \intO \sigmaa_{\by,1}(x)\nabla u(x)\dx
\end{equation*}
with $\sigmaa_{\by,1}$ as defined in \eqref{eq:sigmaat12}.

Before turning to $T_3$ we first note that, since $u$ is piecewise linear,
\begin{equation*}
\begin{split}
u(x) =~& u_j + \frac{x-y_j}{y_j-y_{j-1}}(u_j-u_{j-1}) = u_j + (x-y_{j})\nabla
u(x) \ \ \text{for}\ x\in Q_{j}=(y_{j-1},y_j),\\ 
u(x) =~& u_j + \frac{x-y_j}{y_{j+1}-y_{j}}(u_{j+1}-u_{j}) =
u_j+(x-y_j)\nabla u(x)\ \ \text{for}\ x\in Q_{j+1}=(y_j,y_{j+1}).
\end{split}
\end{equation*}
Hence, $T_3$ in the above equation \eqref{eq:T123} can be written as
\begin{align*}
 T_3 =&\ \veps\sum_{j=-N-1}^N\intO \phi(x)\nabla\delta_\varepsilon(x-y_j)(u(x) - u_j )\dx\\
=&\ \veps\sum_{j=-N-1}^N \intO\phi(x) \nabla\delta_\varepsilon(x-y_j)(x-y_j)\nabla u(x)  \dx 
=\ \veps\intO \sigmaa_{\by,2}(x) \nabla u\dx,
\end{align*}
with $\sigmaa_{\by,2}$ as defined in \eqref{eq:sigmaat12}, which concludes the proof.
\end{proof}

\begin{remark}
1. In more than one space dimension the above calculations can be generalized if a triangular,
respectively, tetrahedral mesh with the atomic positions as nodes is constructed. For example, this
leads to
\begin{equation*}
\begin{split}
\sigmaa_{\by,1}(x) =&\,
\bigl(-\tfrac{1}{2}\veps^2|\nabla\phi|^2-\tfrac{1}{2}m^2\phi^2+\rho_{\by}\phi\bigr)\,\mathrm{id}\ + 
\veps^2\nabla\phi\otimes\nabla\phi.
\end{split}
\end{equation*}

2. A closer look at the calculations in the proof of Lemma
\ref{lemma:weakformulation} shows that the weak form can be obtained
for semilinear models $-\veps^2\Delta\phi +F'(\phi) = \rho_\by$ with
any convex function $F$. Even a fourth-order model of the form
$\veps^4\Delta^{\!2}\phi-\veps^2\Delta\phi+F'(\phi) = \rho_\by$ admits
a similar weak formulation.
\end{remark}

\medskip As already suggested in the introduction the Green's function
for the differential operator $-\veps^2\Delta +m^2 {\rm id}$ acting on
functions defined on $\R$ is given by
\begin{equation}
 G_\veps(x) =  \frac{1}{2\veps m}\onept\e^{-\tfrac{m}{\veps}|x|}.
\label{eq:greensfunction}
\end{equation}
We therefore get the following explicit formulas for the function
values $\phi(x)$ and $\nabla\phi(x)$ for $x\in\Omega$.

\begin{proposition}
\label{prop:phiexpression}
Let $\by\in\mc{Y}$ and let
$\phi=\arg\min_{\varphi\in\mrm{H}^1_\#(\Omega)} I(\varphi,\by)$ be the
corresponding interaction field. Then, for every $x\in\Omega$,
\begin{align}
 \phi(x) =~&  \int_\R G_\veps(x-z)\rho_\by(z)\dz =
\frac{1}{2m} \sum_{k\in\Z} \int_\R\delta_\veps(z-y_k)\, \e^{-\tfrac{m}{\veps}|x-z|}\dz,
\label{eq:phiGreensf}\\
 \nabla\phi(x) =~& \int_\R G_\veps(x-z)\nabla\rho_\by(z)\dz =\frac{1}{2m} \sum_{k\in\Z}
\int_\R\nabla\delta_\veps(z-y_k)\,\e^{-\tfrac{m}{\veps}|x-z|}\dz.
\label{eq:dphiGreensf}
\end{align}
\end{proposition}

\begin{proof} 
  The proof of this proposition is similar to the one of
  \cite[Thm. 2.1]{Evans}; see also \cite[Prop. 2.4]{BLthesis}.
\end{proof}

The following is a consequence of the simple exponential form of the
Yukawa potential and some elementary properties of the exponential
function in one dimension. Let $y_i,y_j\in\R$ satisfy
$y_j>y_i+\veps\vsig$, so that the supports of particle densities
representing the atoms $i$ and $j$ do not intersect. Then,
\begin{align}
  \notag
\int_\R\int_\R \delta_\veps(z-y_j)\onept\e^{-\tfrac{m}{\veps}|z-x|}\delta_\veps(x-y_i)\dx\dz =~& 
\int_\R\int_\R
\delta_\veps(z-y_j)\onept\e^{-\tfrac{m}{\veps}(z-x)}\delta_\veps(x-y_i)\dx\dz\\
\notag
& \hspace{-4cm} = \e^{-\tfrac{m}{\veps}(y_j-y_i)}
\int_\R\e^{-\tfrac{m}{\veps}(z-y_j)}\delta_\veps(z-y_j)\dz 
\cdot\int_\R
\e^{-\tfrac{m}{\veps}(y_i-x)}\delta_\veps(y_i-x)\dx\\
\label{eq:mueffcalculation}
& \hspace{-4cm} = \ceff^2\onept \e^{-\tfrac{m}{\veps}(y_j-y_i)},
\end{align}
where we have defined
\begin{equation*}
 \ceff = \int_\R \delta_\veps(x)\e^{-\tfrac{m}{\veps}x}\dx = \int_\R
\delta_\veps(x)\e^{\tfrac{m}{\veps}x}\dx = \int_\R \delta_1(x)\e^{mx}\dx.
\end{equation*}
Although we will frequently use this property, it is not essential for our reasoning. It merely
makes some calculations more convenient.

\section{Dirichlet Boundary Conditions}
\label{sec:SMDirBCs}

\begin{figure}
\begin{center}
\includegraphics[width=.75\linewidth]{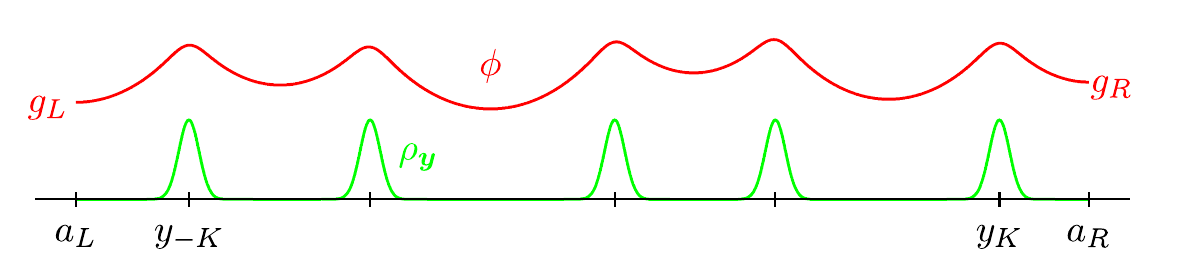}
\caption{The atomistic model in the domain $\Omega_{\aLR}$ with Dirichlet boundary
conditions $g = [g_L \,\, g_R]^{\rm T}$.}\label{fig:QCDirichlet}
\end{center}
\end{figure}

In this section we consider a version of the model
\eqref{eq:Eatomistic} in the domain $\Omega_\aLR=(\aL,\aR)\subset\R$
subject to Dirichlet instead of periodic boundary conditions.  This
concept will be used later on, for the formulation of a/c methods, as
the atomistic subproblem. We set $\aLR = [\aL\,\,\aR]^{\rm T}\in\R^2$ and
$\Delta\aLR = \aR-\aL$. Throughout Section \ref{sec:SMDirBCs} we
think of $\by=(y_{-K},\ldots,y_K)$ as an ordered element of
$\Omega_\aLR^{2K+1}$ such that $\aL<y_{-K}<\cdots<y_K<\aR$. The
particle density $\rho_{\by}$ is canonically defined by
\begin{equation*}
\rho_\by = \veps\sum_{j=-K}^K\delta_\veps(\cdot-y_j).
\end{equation*}
For simplicity we assume that the $y_j$ are separated and lie well
inside $\Omega_\aLR$ in the sense that
$\supp\rho_\by\cap\partial\Omega_\aLR = \emptyset$ or, equivalently,
\begin{equation}
  \label{eq:dircase_sep_assmpt}
  \begin{split}
    & y_i' \geq \vsig, \text{ for } i = -K+1, \dots, K, \\
    & \aR - y_K > \veps\vsig/2 \quad \text{and} \quad
    y_{-K} - \aL > \veps\vsig/2.
  \end{split}
\end{equation}

We impose the following boundary conditions on the resulting field
$\phi:\Omega_\aLR\rightarrow\R$:
\begin{equation*}
\phi(\aL) = g_L,\qquad \phi(\aR) =  g_R;
\end{equation*}
i.e., $\phi|_{\partial\Omega_{\aLR}} = g$ with $g =
[g_L\,\,g_R]^{\rm T}\in\R^2$. The interaction potential
$\E_{\aLR,g}:\Omega_\aLR^{2K+1}\rightarrow\R$ is defined by
\begin{equation}
 \E_{\aLR,g}(\by) =- \min_{\tworow{\varphi\in
\mrm{H}^1(\Omega_\aLR)}{\varphi|_{\partial\Omega_\aLR=g}}} I_{\aLR}(\varphi,\by),
\label{eq:Ebdd}
\end{equation}
where
$I_{\aLR}:\mrm{H}^1(\Omega_{\aLR})\times\Omega_{\aLR}^{2K+1}\rightarrow
\R$ is given by
\begin{equation}
 I_\aLR(\varphi,\by) = \int_{\aL}^{\aR}\bigl(\tfrac{1}{2}\varepsilon^2|\nabla\varphi|^2 +
\tfrac{1}{2} m^2\varphi^2\bigr)\dx - \int_{\aL}^{\aR} \rho_{\by}\varphi\dx.\label{eq:I_a}
\end{equation}
For given $\by$ the minimizer $\phi$
is the weak
solution to 
\begin{equation}
\begin{split}
 -\veps^2\Delta\phi + m^2\phi =~& \rho_{\by}\quad \text{in } \Omega_{\aLR}, \\ 
\phi|_{\partial\Omega_{\aLR}} =~& g.
\end{split}
\label{eq:phiboundedequation}
\end{equation}

We will frequently use the decomposition
\begin{equation}
 \phi = \phi_0+\xi_{\aLR,g},
\label{eq:phi_Dir_additive}
\end{equation}
where $\phi_0\in\mrm{H}^1_0(\Omega_\aLR)$ and $\xi_{\aLR,g}\in
\mrm{H}^{1}(\Omega_\aLR)$, respectively, solve the boundary-value problems
\begin{equation*}
\begin{split}
 -\veps^2\Delta\phi_0+ m^2\phi_0 &~= \rho_{\by}\quad \text{in } \Omega_{\aLR},\\
\phi_0|_{\partial\Omega_{\aLR}} &~= 0
\end{split}
\end{equation*}
and
\begin{equation}
  \label{eq:eqn_xiag}
\begin{split}
-\veps^2\Delta\xi_{\aLR,g} + m^2\xi_{\aLR,g} =~& 0\quad \text{in } \Omega_{\aLR}, \\ 
\xi_{\aLR,g}|_{\partial\Omega_{\aLR}}=~& g.
\end{split}
\end{equation}
This last boundary-value problem can be solved explicitly, which
yields the following lemma.

\begin{lemma}
  \label{th:lemma_xiag}
  The solution $\xi_{a,g}$ of \eqref{eq:eqn_xiag} is given by
  \begin{equation}
    \xi_{\aLR,g}(x) = c_L(\aLR,g)\e^{-\frac{m}{\veps} (x-\aL)} + c_R(\aLR,g)\e^{-\frac{m}{\veps}
      (\aR-x)},
    \label{eq:xi_g}
  \end{equation}
  where the coefficients $c_L(\aLR,g)$ and $c_R(\aLR,g)$ are given by
  \begin{equation}
    c(a,g) = \begin{bmatrix} c_L(a,g)\\c_R(a,g)\end{bmatrix} = 
    \begin{bmatrix}
      1 & \tau \\ \tau & 1
    \end{bmatrix}^{-1}
    \begin{bmatrix} g_L\\g_R
    \end{bmatrix}=: T_{\aLR}^{-1} \Didot g
    \label{eq:cofg}
  \end{equation}
  and we have defined $\tau = \exp(-\tfrac{m}{\veps}\Delta \aLR)$.
\end{lemma}

\medskip

Note that, for $\Delta\aLR \gg \veps$, $\tau$ is exponentially small;
hence we will often neglect terms of that order of magnitude. We will
write $\mathcal{O}(\tau)$ for a quantity or function that is
(uniformly) bounded above by $C\tau$ in modulus, where $C$ is
independent of $\veps$ and $\Delta\aLR$. For example, we have $c(\aLR,
g) = g + \mathcal{O}(\tau)$.

Next, we compute the derivative of $\E_{\aLR,g}$ with respect to the
atomic coordinates.  For these derivatives, we obtain a ``weak''
formulation of the same shape as in the periodic case (see Proposition
\ref{prop:periodicderivatives}).

If $\by' > 0$, then we denote by $\calS(\by\cup\aLR)$ the set of
continuous, piecewise affine functions over the mesh given by the
nodes $\aL,y_{-K},\ldots,y_K,\aR$. Moreover,
$\calS_0(\by\cup\aLR)=\calS(\by\cup\aLR)\cap\mrm{H}^1_0(\Omega_\aLR)$.

\begin{proposition}
\label{lem:weakformbounded}
Let $a,g\in\R^2$, $a_L < a_R$; then $\E_{a,g}:\mc{Y}\rightarrow\R$
defined by \eqref{eq:Ebdd} is continuously Fr\'{e}chet differentiable
at $\by$.

(i) The components of the first derivative are given by
\begin{equation}
D_{y_j} \E_{\aLR,g}(\by)  =-
\veps\int_{\Omega_\aLR}\nabla\delta_\varepsilon(x-y_j)\phi(x)\dx
\qquad \text{for } i = -K, \ldots, K.
\label{eq:nablaVbdd}
\end{equation}

(ii) Let $\bu\in\U$ be a test vector, $u\in\Sm_{0}(\by\cup\aLR)$ its
interpolant, and let $\min\by' \geq \vsig$; then
\begin{equation}
  \label{eq:DEdir_u0_stress}
  D_{\by}\E_{\aLR,g}(\by)\Didot \bu = \int_{\Omega_\aLR} \sigmaa_{\by}(x)\nabla u(x)\dx,
\end{equation}
where $\sigmaa_{\by}$ is given by \eqref{eq:sigmaat12}.
\end{proposition}

\begin{proof}
The derivatives with respect to the coordinates $\by$ are easy to calculate along the same
lines as in the proof of Proposition \ref{prop:periodicderivatives}. The weak formulation
can be obtained as in the periodic case (Lemma \ref{lemma:weakformulation}) using the fact that the
interpolant $u$ vanishes on $\partial\Omega_\aLR$.
\end{proof}

\begin{remark}
  We point out that, in general,
  \begin{equation*}
    \E_{\aLR,g}(\by) \neq \frac{1}{2}\int_{\Omega_\aLR} \rho_\by\phi\dx.
  \end{equation*}
  However, we will see below that $\E_{\aLR,g}(\by)$ can be written as
  the sum of a boundary data contribution and a term that is
  independent of $g$.
\end{remark}

\medskip With a view to the subsequent derivation of a/c methods we
will from now on interpret $\aLR$ and $g$ as arguments to
$\E_{\aLR,g}$ rather than fixed parameters entering its definition.
We consider the map $\Omega_{\aLR}^{2K+1}\times
\R^2\times\R^2\rightarrow \R$, $(\by,\aLR,g)\mapsto \E_{\aLR,g}(\by)$,
and derive the derivatives of this map with respect to the boundary
$\aLR$ and the boundary data $g$.

\subsection{Dependence on the boundary positions}
When formulating a/c methods in Section \ref{Sec:QCCoupling} we will
let the boundary $\aLR$ of the atomistic subdomain depend on the
configuration $\by$. It is therefore necessary to understand the
dependence of the energy $\E_{\aLR,g}(\by)$ on $\aLR$.  Our main
result is that the derivative $D_{\aLR}\E_{\aLR,g}(\by)$ can be
combined with $D_\by\E_{\aLR,g}(\by)$ into a weak formulation
reminiscent of \eqref{eq:per_weakform}. This will be a central
building block for a/c methods.

\begin{proposition}
\label{cor:Ebdd_weakform}
Suppose that $\by\in\mc{Y}$, $\min\by' \geq \vsig$. Let
$h=[h_L\,\,h_R]^{\rm T}\in\R^2$ and $\bu=(u_{-K},\ldots,u_K)\in\R^{2K+1}$ be
test vectors, and let $u\in \Sm(\by\cup\aLR)$ denote the interpolant
of $\bu$ and $h$ in the sense that
\begin{equation*}
u(\aL)=h_{L},\ \ u(\aR) = h_R,\ \ \text{and} \ \ u(y_j) = u_j \quad \forall j\in\{-K,\ldots,K\}.
\end{equation*}
Then,
\begin{equation*}
  D_{\aLR}\E_{\aLR,g}(\by)\Didot h + D_{\by}\E_{\aLR,g}(\by)\Didot \bu
= \int_{\Omega_\aLR} \sigmaa_{\by}(x)\nabla u(x)\dx.
 \end{equation*}
\end{proposition}

\begin{proof}
This is a direct consequence of the following two lemmas.
\end{proof}

In the first auxiliary lemma we compute the derivative of
$\E_{\aLR,g}(\by)$ with respect to $\aLR=[\aL\,\,\aR]^{\rm T}$ while keeping
the relative distances between the atoms constant. In other words we
consider the change in $\E_{\aLR,g}(\by)$ when the whole domain
$\Omega_a$ is stretched with the atom positions following this
stretching. For $\by\in\Omega_\aLR^{2K+1}$ let
$\bs{X}=(X_{-K},\ldots,X_K)\in(0,1)^{2K+1}$ be given by $y_j = \aL +
\Delta a X_j$ for all $j\in\{-K,\ldots,K\}$. For fixed $g, \bs{X}$ we
define
\begin{align}
  \notag
  \tilde\E(\aLR) \coloneqq~& \E_{\aLR,g}(\aL + (\aR-\aL) \bs{X}), \quad
  \text{and} \\
  \notag
  \tilde D_{\aR}\E_{\aLR,g}(\by) \coloneqq~& D_{\aR} \tilde\E(\aLR).
\end{align}
(We understand $\aL + (\aR-\aL)\bs{X}$ in a componentwise manner:
$(\aL + \Delta \aLR \bs{X})_j = \aL + \Delta \aLR X_j$ for all
$j\in\{-K,\ldots,K\}$.) The derivative $\tilde
D_{\aL}\E_{\aLR,g}(\by)$ is defined analogously.

\begin{lemma}
\label{lemma:DEyN}
Let $\by\in\Omega_\aLR^{2K+1}$ satisfy \eqref{eq:dircase_sep_assmpt}; then
\begin{equation}
\label{eq:DtildeEyN}
- \tilde{D}_{\aL} \E_{\aLR,g}(\by) = \tilde D_{\aR}\E_{\aLR,g}(\by)  = \frac{1}{\Delta
\aLR}\int_{\Omega_\aLR}\sigmaa_{\by}(x)\dx.
\end{equation}
\end{lemma}

\begin{proof} 
  We fix $\bs{X}$ and let $\bs{\eta}(\aLR) \coloneqq \aL + \Delta a
  \bs{X}$. We begin by transforming the problem to the unit interval
  $(0,1)$ using the transformation $x\mapsto X(x) =
  (x-\aL)/(\aR-\aL)$:
\begin{equation}
\label{eq:EaLR_written_out}
\begin{split}
 \tilde{\E}(\aLR) = \E_{\aLR,g}(\bs{\eta}(\aLR))=~&\int_{\Omega_\aLR}\Bigl(
-\tfrac{1}{2}\veps^2|\nabla\phi|^2 -
\tfrac{1}{2}m^2\phi^2 + \rho_{\bs{\eta}(\aLR)}\phi\Bigr)\dx\\
=~& \Delta \aLR\int_{0}^{1}\biggl(-\frac{\veps^2}{2\Deltaa^2}|\nabla\hat\phi|^2
-\frac{m^2}{2}{\hat\phi}^{\hspace{1pt} 2}
+\hat\rho_{\bs{\eta}(\aLR)}\hat\phi\biggr)\dX\!.
\end{split}
\end{equation}
Here, $\hat\phi(X) = \phi(x(X))$ and $\hat{\rho}_{\bs{\eta}(\aLR)}(X)
= \rho_{\bs{\eta}(\aLR)}(x(X))$.  It follows as in Proposition
\ref{prop:periodicderivatives} that, to compute
$D_{\aLR}\tilde{\E}(\aLR)$, it is sufficient to calculate the partial
derivatives of the right-hand side with respect to $\aR$ (the
derivative of $\phi$ or $\hat\phi$ with respect to $\aR$ does not
appear since $\phi$ is a minimizer of $I_{\aLR}(\cdot,\by)$). This
leads to
\begin{equation*}
\begin{split}
D_{\aR}\tilde{\E}(\aLR) =~&\int_{0}^{1}\biggl(-\frac{\veps^2}{2\Deltaa^2}|\nabla\hat\phi|^2
-\frac{m^2}{2}{\hat\phi}^{\hspace{1pt} 2}
+\hat\rho_{\bs{\eta}(\aLR)}\hat\phi\biggr)\dX + \Delta \aLR
\int_{0}^{1}\frac{\veps^2}{\Deltaa^3}|\nabla\hat\phi|^2 \dX\!\\
&~+\Delta \aLR \int_0^1\hat\phi D_{\aR}\hat{\rho}_{\bs{\eta}(\aLR)}\dX.
\end{split}
\end{equation*}
Transforming the first two integrals on the right-hand side back to the interval $\Omega_{\aLR}$ we
arrive at
\begin{equation*}
 \frac{1}{\Delta \aLR} \E_{\aLR,g}(\by) + \frac{\veps^2}{\Delta \aLR}\int_{\Omega_\aLR}
|\nabla\phi|^2\dx = \frac{1}{\Delta \aLR}\int_{\Omega_\aLR}\sigmaa_{\by,1}(x)\dx,
\end{equation*}
where $\sigmaa_{\by,1}$ was given in \eqref{eq:sigmaat12}.

It remains to differentiate $\hat\rho_{\bs{\eta}(\aLR)}$ with respect
to $\aR$.  By the definition of the transformation $x\mapsto X(x)$ we
have
\begin{align*}
 D_{\aR}\hat\rho_{\bs{\eta}(\aLR)}(X) =&\, \veps D_{\aR}\sum_{j=-K}^K\delta_\veps\bigl((\aR-\aL)
(X-X_j)\bigr)\\
  =&~ \veps\sum_{j=-K}^{K}(X-X_j)\nabla\delta_\veps\bigl( (\aR-\aL)(X-X_j)\bigr).
\end{align*}
Using $\Delta \aLR(X-X_j) = (x-y_j)$ we therefore get
\begin{align*}
 \Delta \aLR  \int_{0}^{1}\hat\phi D_{\aR}\hat\rho_{\bs{\eta}(\aLR)}\dX =~& \frac{\veps}{\Delta
\aLR}\sum_{j = -K}^{K}\int_{\Omega_\aLR}(x-y_j)\nabla\delta_\veps(x-y_j)\phi(x)\dx\\
 =~&\frac{1}{\Delta \aLR}\int_{\Omega_\aLR}\sigmaa_{\by,2}(x)\dx
\end{align*}
with $\sigmaa_{\by,2}(x)$ as given in \eqref{eq:sigmaat12}.

To see that $D_{\aL}\tilde{\E} = - D_{\aR}\tilde{\E}$ we simply note
that $\E(\aLR)$ depends only on $\Delta\aLR$, which can be seen from
\eqref{eq:EaLR_written_out} and the definition of
$\hat{\rho}_{\bs{\eta}(\aLR)}(X)$.
\end{proof}

We define $\theta_R\in\Sm(\by\cup\aLR)$ to be the piecewise linear
function with
\begin{equation*}
 \theta_R(\aR) = 1,\quad \theta_R(\aL) = 0,\quad \theta_R(y_j)=0\  \ \text{for all }
j\in\{-K,\ldots,K\}.
\end{equation*}
The function $\theta_L\in\Sm(\by\cup\aLR)$ is defined analogously.

\begin{lemma}
  Let $\by\in\Omega_\aLR^{2K+1}$ satisfy
  \eqref{eq:dircase_sep_assmpt}; then, the derivatives of
  $\E_{\aLR,g}(\by)$ with respect to $\aL$, $\aR$ (for fixed $\by$ and
  $g$) satisfy
\begin{equation*}
\begin{split}
D_{\aL} \E_{\aLR,g}(\by) =~& \int_{\Omega_\aLR}\sigmaa_{\by}(x)\nabla\theta_{L}(x)\dx,\\ 
D_{\aR} \E_{\aLR,g}(\by) =~& \int_{\Omega_\aLR}\sigmaa_{\by}(x)\nabla\theta_{R}(x)\dx.
\end{split}
\end{equation*} 
\end{lemma}

\begin{proof}
Let $\Theta_{R}$ be the affine function defined on $\Omega_\aLR$ with $\Theta_{R}(\aL) = 0$,
$\Theta_{R}(\aR) = 1$. Since $\nabla\Theta_R(x) = \tfrac{1}{\Delta\aLR}$, Lemma \ref{lemma:DEyN}
yields
\begin{equation}
 \tilde D_{\aR}\E_{\aLR,\by}(\by) = \int_{\Omega_\aLR} \sigmaa_{\by}\nabla\Theta_R\dx =
\int_{\Omega_\aLR} \sigmaa_{\by}\nabla(\Theta_R-\theta_R)\dx + \int_{\Omega_\aLR}
\sigmaa_{\by}\nabla\theta_R\dx.
\label{eq:Thetatheta}
\end{equation}
Now, we have $\Theta_R-\theta_R\in\Sm_0(\by\cup\aLR)$ and hence, by Proposition
\ref{lem:weakformbounded},
\begin{equation}
 \int_{\Omega_\aLR} \sigmaa_{\by}(x)\nabla(\Theta_R-\theta_R)\dx = \sum_{j=-K}^K
D_{y_j}\E_{\aLR,g}(\by) \Theta_R(y_j).
\label{eq:Thetatheta2}
\end{equation}
However, $\tilde D_{\aR}\E_{\aLR,g}(\by)$ was defined as the derivative with respect to $\aR$, while
the relative distances of the atoms are kept constant. This can be formulated as
\begin{equation*}
 \tilde D_{\aR}\E_{\aLR,g}(\by) = D_{\aR}\E_{\aLR,g}(\by) + \sum_{j=-K}^K
D_{y_j}\E_{\aLR,g}(\by) \Theta_R(y_j).
\end{equation*}
Inserting this into \eqref{eq:Thetatheta} and using \eqref{eq:Thetatheta2} then gives
\begin{equation*}
  \int_{\Omega_\aLR} \sigmaa_{\by}(x)\nabla\theta_R\dx = D_{\aR}\E_{\aLR,g}(\by).
\end{equation*}
Similarly, we can show the expression stated for $D_{\aL}\E_{\aLR,g}(\by)$.
\end{proof}

\subsection{Dependence on the boundary conditions}
Next, we compute the derivative of $\E_{\aLR,g}(\by)$ with respect to
the boundary conditions $g$ when the configuration $\by$ and the
boundary $\aLR$ are kept fixed. 
We define
\begin{equation}
\begin{split}
\gamma_{L}(\by,\aLR) =~& 2\int_{\Omega_\aLR}\!
\rho_{\by}(x)G_\veps(x-\aL)\dx,
\quad \text{and} \quad
\gamma_{R}(\by,\aLR) = 2\int_{\Omega_\aLR}\! \rho_{\by}(x)G_\veps(\aR-x)\dx.
\end{split}
\label{eq:gammaLgammaR}
\end{equation}

\begin{lemma}
\label{lemma:depbdrycond}
The partial derivative of $\E_{\aLR,g}(\by)$ with respect to $g$ is given by:
\begin{equation*}
D_{g}\E_{\aLR,g}(\by) = -m\veps \biggl((1-\tau^2)\begin{bmatrix}c_L(\aLR,g)\\c_R(\aLR,g)
\end{bmatrix} -\begin{bmatrix}\gamma_{L}(\by,\aLR) \\ \gamma_{R}(\by,\aLR)\end{bmatrix}
 \biggr)^{\!T}\!\!\cdot T_{\aLR}^{-1},
\end{equation*}
where $T_{\aLR}$, $c(\aLR,g)=[c_L(\aLR,g)\,\,c_R(\aLR,g)]^{T}$ and
$\tau = \e^{-\tfrac{m}{\veps}\Delta a}$ are defined in Lemma
\ref{th:lemma_xiag}.
\end{lemma}

\begin{proof}
  Throughout the proof we suppress the arguments of $\gamma_L$,
  $\gamma_R$, and $c$ for ease of readability.  We recall the additive
  decomposition $\phi=\phi_0+\xi_{\aLR,g}$ from
  \eqref{eq:phi_Dir_additive}. From $\phi_0\in\Hne$, and from the
  equation $-\veps^2\Delta\xi_{\aLR,g} + m^2\xi_{\aLR,g} = 0$ it
  follows that $\veps^2(\nabla
  \xi_{\aLR,g},\nabla\phi_0)+m^2(\xi_{\aLR,g},\phi_0) = 0$. Hence, a
  short calculation shows that the energy $\E_{\aLR,g}(\by)$ can be
  rewritten as
  \begin{equation}
    \E_{\aLR,g}(\by) = -I_{\aLR}(\phi,\by) = -I_{\aLR}(\phi_0,\by) -
    I_{\aLR}(\xi_{\aLR,g},\by).
    \label{eq:Eintbdry}
  \end{equation}
  The first term on the right-hand side does not depend on the
  boundary conditions $g$ and the second term is known explicitly:
  using $-\veps^2\Delta\xi_{\aLR,g} + m^2\xi_{\aLR,g} = 0$,
  integration by parts, and the explicit formula \eqref{eq:xi_g} for
  $\xi_{\aLR,g}$, we obtain
  \begin{align*}
    I_{\aLR}(\xi_{\aLR,g},\by) =& \int_{\Omega_\aLR}\tfrac{1}{2}\bigl(
    \veps^2|\nabla\xi_{\aLR,g}|^2+m^2\xi_{\aLR,g}^2\bigr)\dx - \int_{\Omega_\aLR}
    \rho_{\by}\xi_{\aLR,g}\dx\\ 
    =&\, \frac{\veps^2}{2}\bigl( -\xi_{\aLR,g}(\aL)\nabla
    \xi_{a,g}(\aL)+\xi_{a,g}(\aR)\nabla\xi_{\aLR,g}(\aR)\bigr)- \int_{\Omega_\aLR}
    \rho_{\by}\xi_{\aLR,g}\dx\\
    =&\, \frac{\veps m}{2}\bigl(c_L^2+c_R^2\bigr)\bigl(1-\e^{-2\tfrac{m}{\veps}\Deltaa}\bigr)-
    \int_{\Omega_\aLR}
    \rho_{\by}\bigl(c_L\e^{-\frac{m}{\veps} (x-\aL)} + c_R\e^{-\frac{m}{\veps} (\aR-x)}\bigr)\dx\\
    =&\,m\veps \biggl(\frac{c_L^2+c_R^2}{2}\bigl(1-\tau^2\bigr)- \frac{2}{2m\veps}\int_{\Omega_\aLR}
    \rho_{\by}\bigl(c_L\e^{-\frac{m}{\veps} (x-\aL)} + c_R\e^{-\frac{m}{\veps}
      (\aR-x)}\bigr)\dx\biggr)\\
    =&\,m\veps \biggl(\frac{c_L^2+c_R^2}{2}\bigl(1-\tau^2\bigr)- \bigl( c_L \gamma_{L} 
    +c_R \gamma_{R}\bigr)\biggr).
  \end{align*}
  Here we have used the Green's function $G_\veps$ from
  \eqref{eq:greensfunction}. Differentiating this expression with
  respect to $c_L$ and $c_R$ and applying the chain rule with $D_gc =
  T_{\aLR}^{-1}$ yield the result.
\end{proof}

\begin{remark}
  \label{remark:refl}
  1. We remark that $D_{g}\E_{\aLR,g}(\by)=0$ if and only if
  $c_L(\aLR,g) =\gamma_L(\by,\aLR)/(1-\tau^2)$ and $c_R(\aLR,g)
  =\gamma_R(\by,\aLR)/(1-\tau^2)$. According to \eqref{eq:cofg} this
  corresponds to the ``optimal'' boundary conditions
  \begin{equation}
    g_{L}^* = \frac{1}{1-\tau}\frac{\gamma_{L}+\tau
      \gamma_{R}}{1+\tau},\qquad \text{and}\qquad
    g_{R}^* = \frac{1}{1-\tau}\frac{\tau \gamma_{L}+\gamma_{R}}{1+\tau}.
    \label{eq:g*}
  \end{equation}
  That is, the boundary conditions are weighted averages of the values
  $\frac{1}{1-\tau}\gamma_{L}$ and $\frac{1}{1-\tau}\gamma_R$.

  2. As can be seen from Lemma \ref{lemma:depbdrycond} the boundary
  data contribution $I_{\aLR}(\xi_{\aLR,g},\by)$ to the energy
  $\E_{\aLR,g}(\by)$ is quadratic in $g$. For fixed configuration
  $\by$ and domain $\Omega_{\aLR}$ the boundary conditions
  $g=g^*(\by,\aLR)$ minimize the boundary data contribution
  $I_{\aLR}(\xi_{\aLR,g},\by)$ to the energy $\E_{\aLR,g}(\by)$. This
  is equivalent to minimizing $I_{\aLR}(\cdot,\by)$ over
  $\mrm{H}^1(\Omega_{\aLR})$ and therefore leads to homogeneous
  Neumann boundary conditions for $\phi$ on $\partial\Omega_{\aLR}$.

  3. If $\Delta \aLR \gg \veps$, i.e., $\tau \ll 1$, then we have
  $\gamma_{L/R} = g_{L/R}^* + \mathcal{O}(\tau)$, and hence we can
  simplify
  \begin{align}
    \label{eq:Isimpleexpression}
    I_{\aLR}(\xi_{\aLR,g},\by) =~& m\veps
    \bigl(\smfrac{1}{2}\bigl(g_L^2+g_R^2\bigr)- \bigl(g_L g^*_{L} +g_R
    g^*_{R}\bigr)\bigr) +\mc{O}(\veps\tau), \quad
    \text{and} \\
    \notag
    D_{g}\E_{\aLR,g}(\by) =~& m \veps (g^* - g)
    + \mathcal{O(\veps\tau)}. \qedhere
  \end{align}
\end{remark}

A useful auxiliary result for the analysis of a/c methods is the global
Lipschitz continuity of the field $\phi$ with respect to variations in
the boundary conditions $g$. 

\begin{lemma}
  \label{lemma:maxprinc}
  Let $\phi_1,\phi_2\in\mrm{H}^1(\Omega_{\aLR})$ be minimizers of
  $I_{\aLR}(\cdot,\by)$ subject to the boundary conditions
  $g_1\in\R^2$, respectively, $g_2\in\R^2$. Then,
  \begin{equation*}
    \begin{split}
      |\phi_1(x)-\phi_2(x)| \leq~&
      \sqrt{2} \, \snorm{T_{\aLR}^{-1}(g_1-g_2)}  \, 
      \e^{-\tfrac{m}{\veps} d_\aLR(x)}, \quad \text{and} \\
      \veps |\nabla\phi_1(x)-\nabla\phi_2(x)| \leq~& 
      \sqrt{2} \,m\, \snorm{T_{\aLR}^{-1}(g_1-g_2)} \, \e^{-\tfrac{m}{\veps} d_\aLR(x)},
    \end{split}
  \end{equation*}
  where $d_\aLR(x) := \min(x - \aL, \aR - x)$ denotes the distance to
  the boundary of $\Omega_{\aLR}$, for $x \in \Omega_{\aLR}$.
\end{lemma}
\begin{proof}
  We write both functions in the form $\phi_i = \phi_0 +
  \xi_{\aLR,g_i}$, $i\in\{1,2\}$. For $i = 1,2$, let $c_i = T_a^{-1}
  g_i$ be the respective coefficients entering $\xi_{\aLR,g_i}$; then 
  \begin{equation*}
    \begin{split}
      |\phi_1(x) - \phi_2(x)| =~& |\xi_{\aLR,g_1}(x)-\xi_{\aLR,g_2}(x)|
      \leq |c_{1,L}-c_{2,L}|\, \e^{-\tfrac{m}{\veps}(x-\aL)}+|c_{1,R}-c_{2,R}|\,
      \e^{-\tfrac{m}{\veps}(\aR - x)}.
   \end{split}
  \end{equation*}
  This immediately yields the first bound.  The bound for the
  derivatives is obtained similarly.
\end{proof}

\subsection{A special case}
\label{sec:specialcase}
We now take a closer look at the interaction potential $\E_{\aLR,g}$
from \eqref{eq:Ebdd} with the $\by$-dependent boundary conditions
$g=g^*(\by,\aLR)$ defined in Remark \ref{remark:refl}. 

\begin{proposition}
Let $\by\in\Omega^{2K+1}_{\aLR}$. Then,
\begin{equation}
\begin{split}
 \E_{\aLR,g^*(\by,\aLR)}(\by)
=~&\frac{1}{4m\veps}\int_{\Omega_\aLR}\!\int_{\Omega_\aLR}\rho_{\by}(x)
\e^{-\tfrac{m}{\veps}|x-z|}
\rho_{\by}(z)\dz\dx +\tau M_\tau(\gamma_L,\gamma_R)\\
&+\frac{1}{4m\veps}\int_{\Omega_\aLR}\!\int_{\Omega_\aLR}\rho_{\by}(x)
\bigl(\e^{-\tfrac{m}{\veps}(2\aR-x-z)}+\e^{-\tfrac{m}{\veps}(x+z-2\aL)}\bigr)\rho_{\by}(z)\dz\dx,
\end{split}
\label{eq:Eg*}
\end{equation}
where $M_\tau(\gamma_L,\gamma_R)$ depends quadratically on $\gamma_L$ and $\gamma_R$.
\label{prop:Eg*}
\end{proposition}

\medskip Expression \eqref{eq:Eg*} can be interpreted as the energy of
the atoms represented by $\by$ interacting with each other plus the
interaction with mirror atoms outside $\Omega_\aLR$. This mirror
interaction was introduced by means of the boundary conditions $g =
g^*$.

For the proof of the proposition it is convenient to use an explicit
formula for the function values of
$\phi_0\in\mrm{H}^1_0(\Omega_{\aLR})$ from the decomposition
\eqref{eq:phi_Dir_additive}.  By Proposition
\ref{prop:phiexpression}, the Green's function for the equation
$-\veps^2\Delta\phi +m^2\phi = \rho_{\by}$ in $\R$ is given by
$G_\veps(x,y) =
\frac{1}{2m\veps}\onept\e^{-\frac{m}{\veps}|x-y|}$. 
We will now construct the Green's function $G_{\veps,\aLR}$ for the
operator $-\veps^2\Delta+m^2 {\rm id}$ subject to homogeneous Dirichlet
conditions on $\partial\Omega_\aLR$.


\begin{lemma}
  \label{lemma:Greensfunctionbounded}
  Let $\phi_0\in\mrm{H}^1_0(\Omega_{\aLR})$ satisfy
  $-\veps^2\Delta\phi_0+m^2\phi_0= \rho_{\by}$ in
  $\Omega_{\aLR}$. Then,
  \begin{equation}
    \phi_0(x) = \int_{\Omega_\aLR} G_{\veps,\aLR}(x,z)\rho_{\by}(z)\dz\quad \forall x\in\Omega_{\aLR},
    \label{eq:phiphi0xiGreens}
  \end{equation}
  where $G_{\veps,\aLR} = G_{\veps,\aLR}^{(1)} + \tau
  G^{(2)}_{\veps,\aLR}$, with $G_{\veps,\aLR}^{(i)}$, $i = 1,2$, given
  by
  \begin{equation*}
    \begin{split}
      G_{\veps,\aLR}^{(1)}(x,z)=~&\frac{1}{2m\veps}\Bigl(\e^{-\tfrac{m}{\veps}|x-z|} -
      \e^{-\tfrac{m}{\veps}(x+z-2\aL)} -\e^{-\tfrac{m}{\veps}(2\aR - x-z)}\Bigr),\\
      G_{\veps,\aLR}^{(2)}(x,z)=~& -\frac{1}{2m\veps}\frac{1}{1-\tau^2}\Bigl(
      \tau\e^{-\tfrac{m}{\veps}(x+z-2\aL)}+\tau\e^{-\tfrac{m}{\veps}(2\aR - x-z)}\\
      &\hspace{2.7cm}-\e^{-\tfrac{m}{\veps}(x-z+\aR-\aL)}-
      \e^{-\tfrac{m}{\veps}(z-x+\aR-\aL)}\Bigr).
    \end{split}
  \end{equation*} 
\end{lemma}
\begin{proof} 
  The proof of this result is standard \cite[Chapter 2.2.4]{Evans};
  see also \cite[Lemma 3.10]{BLthesis}.
\end{proof}

We remark that $G_{\veps,\aLR} = G_{\veps,\aLR}^{(1)} +
\mathcal{O}(\tau)$.

\begin{proof}[Proof of Proposition \ref{prop:Eg*}]
 We have already seen in \eqref{eq:Eintbdry} that for any choice of boundary data $g\in\R^2$ the
energy $\E_{a,g}(\by)$ can be written as the sum of two terms
\begin{equation*}
 \E_{\aLR,g}(\by) = -I_{\aLR}(\phi,\by)= -I_{\aLR}(\phi_0,\by) -
I_{\aLR}(\xi_{\aLR,g},\by),
\end{equation*}
where $I_{\aLR}(\phi_0,\by)$ is independent of the boundary conditions.

\emph{Calculation of $I_{\aLR}(\phi_0,\by)$.} Since the function $\phi_0$ is a minimizer
of $I_{\aLR}(\cdot,\by)$ over $\Hne$, we have
with the expression \eqref{eq:phiphi0xiGreens} for $\phi_0(x)$ that
\begin{equation}
I_{\aLR}(\phi_0,\by) = -\frac{1}{2}\int_{\Omega_\aLR} \int_{\Omega_\aLR} \rho_\by \phi_0\dx=
-\frac{1}{2}\int_{\Omega_\aLR}\int_{\Omega_\aLR}
\rho_{\by}(x)G_{\veps,\aLR}(x,z)\rho_\by(z)\dz\dx.
\label{eq:IaLRphi0temp}
\end{equation}
By the definition \eqref{eq:gammaLgammaR} of $\gamma_L$ and $\gamma_R$ we have
\begin{equation}
\begin{split}
 \frac{1}{4m\veps}\int_{\Omega_\aLR}\!\int_{\Omega_\aLR}\rho_{\by}(x)\onept\e^{-\tfrac{m}{\veps}
(2\aR-x-z)}\rho_{\by} (z)\dx\dz
=~& \frac{m\veps}{4} \gamma_R^2,\\
\frac{1}{4m\veps}\int_{\Omega_\aLR}\!\int_{\Omega_\aLR}\rho_{\by}(x)\onept\e^{-\tfrac{m}{\veps}
(x+z-2\aL)}\rho_{\by} (z)\dx\dz
=~& \frac{m\veps}{4} \gamma_L^2,\\
\frac{1}{4m\veps}\int_{\Omega_\aLR}\!\int_{\Omega_\aLR}\rho_{\by}(x)\onept\e^{-\tfrac{m}{\veps}
(z-x+\aR-\aL)}\rho_{\by}
(z)\dx\dz =~& \frac{m\veps}{4}\gamma_L \gamma_R.
\end{split}
\label{eq:gammaLsquared}
\end{equation}
Inserting the expression $G_{\veps,\aLR} = G_{\veps,\aLR}^{(1)} + \tau
G_{\veps,\aLR}^{(2)}$ into \eqref{eq:IaLRphi0temp} and using these
equalities yields
 \begin{equation*}
\begin{split}
I_{\aLR}(\phi_0,\by) =~&
-\frac{1}{2}\int_{\Omega_\aLR}\int_{\Omega_\aLR} \rho_{\by}(x)G_\veps(x,z)\rho_{\by}(z)\dz\dx
+\frac{m\veps}{4} \bigl( \gamma_L^2 +\gamma_R^2\bigr)\\
& + \frac{m\veps}{4}\frac{\tau}{1-\tau^2} \bigl(\tau\gamma_L^2
+\tau\gamma_R^2 -2\gamma_L\gamma_R\bigr).
\end{split}
 \end{equation*}

\emph{Calculation of $I_{\aLR}(\xi_{\aLR,g^*(\by,\aLR)},\by)$.} From Lemma
\ref{lemma:depbdrycond} we know that for general $g \in \R^2$
\begin{equation*}
 I_{\aLR}(\xi_{\aLR,g},\by) =m\veps
\biggl(\frac{c_L^2+c_R^2}{2}\bigl(1-\tau^2\bigr)- \bigl(c_L
\gamma_{L}  +c_R \gamma_{R}\bigr)\biggr).
\end{equation*}
If $g = g^*(\by,\aLR)$, then $c_L = \gamma_L/(1-\tau^2)$ and $c_R =\gamma_R/(1-\tau^2)$ as
seen in Remark \ref{remark:refl}. Hence,
\begin{equation*}
 I_{\aLR}(\xi_{\aLR,g^*(\by,\aLR)},\by) = - \frac{m\veps}{2}\frac{1}{1-\tau^2}
\bigl(\gamma_L^2 + \gamma_R^2\bigr).
\end{equation*}
Isolating the dependence on $\tau$ gives
 \begin{equation}
I_{\aLR}(\xi_{\aLR,g^*(\by,\aLR)},\by) =
-\frac{m\veps}{2}\bigl(\gamma_L^2+\gamma_R^2\bigr)-\frac{m\veps}{2}\frac{\tau^2}
{ 1-\tau^2}\bigl(\gamma_L^2+\gamma_R^2\bigr).
\label{eq:Ixia}
 \end{equation}

\emph{Conclusion.} Adding $-I_{\aLR}(\xi_{\aLR,g^*(\by,\aLR)},\by)$ as just obtained and
$-I_{\aLR}(\phi_0,\by)$ from above we arrive at
\begin{equation}
\begin{split}
 \E_{\aLR,g^*(\by,\aLR)}(\by) =~& \frac{1}{4m\veps}\int_{\Omega_\aLR}\!\int_{\Omega_\aLR}
\rho_{\by}(x)\onept \e^{-\tfrac{m}{\veps}|x-z|} \rho_{\by}(z)\dz\dx +  \frac{m\varepsilon
}{4}\left(\gamma_L^2+\gamma_R^2\right)\\[1mm]
&-\frac{m\veps}{4}\frac{\tau}{1-\tau^2}  \bigl(\tau\gamma_L^2+2 \gamma_L
\gamma_R+\tau \gamma_R^2\bigr).
\end{split}
\label{eq:Eagwithtau}
\end{equation}
Defining $\tau M_\tau(\gamma_L,\gamma_R)$ to be the third term on the right-hand side and applying
\eqref{eq:gammaLsquared} yields \eqref{eq:Eg*}.
\end{proof}

\section{The Cauchy--Born Approximation}
\label{sec:SM_CB}

\begin{figure}
\begin{center}
\includegraphics[width=.7\linewidth]{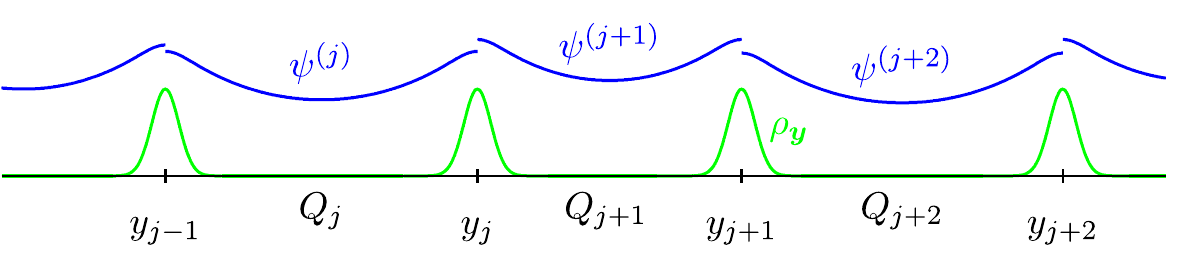}
\caption{The Cauchy--Born approximation: independent periodic problems are
solved on the cells $Q_j = (y_{j-1},y_j)$ leading to locally defined fields
$\psi^{(j)}$.}\label{fig:QCCB}
\end{center}
\end{figure}

The next building block for the design of a/c methods based on the
model \eqref{eq:Eatomistic} is the respective continuum model. Let
$\by\in\mc{Y}$ satisfy $\min\by'>\vsig$. The Cauchy--Born
approximation is obtained by computing the energy of the cells
$Q_j=(y_{j-1},y_j)$ independently from one another, by treating each
of them as part of a homogeneous chain (see Figure \ref{fig:QCCB}).
We define the Cauchy--Born energy of the cell $Q_j$ by
\begin{equation}
 \E_j^{\cb}(\by) = -\min_{\psi\in\mrm{H}^1_{\#}(Q_j)}\biggl(
\int_{Q_j}\bigl(\tfrac{1}{2}\veps^2|\nabla\psi|^2+\tfrac{1}{2}m^2\psi^2\bigr)\dx -
\int_{Q_j}\!\rho_{\by}\psi\dx\biggr).
\label{eq:Ecb}
\end{equation}
Note that this energy only depends on the distance
$(y_j-y_{j-1})$. The minimizer $\psi^{(j)}$ of \eqref{eq:Ecb}
satisfies the equation $-\veps^2\Delta\psi^{(j)} + m^2\psi^{(j)} =
\rho_{\by}$ in $Q_j$ and its $|Q_j|$-periodic extension to $\R$:
\begin{equation}
 -\veps^2\Delta\psi^{(j)} + m^2\psi^{(j)} = \rho_{\by^{(j)}} \quad\text{in }\R.
\label{eq:psijPDE}
\end{equation}
Here we have defined the positions $\by^{(j)}=(y^{(j)}_k)_{k\in\Z}$ of an infinite chain of
equidistant atoms by 
\begin{equation}
  \label{eq:cb:periodic_ext}
y^{(j)}_k = y_j + (k-j)(y_j-y_{j-1})\quad \forall k\in\Z.
\end{equation}
The Cauchy--Born approximation $\E^{\cb}(\by)$ of the atomistic
energy $\E(\by)$ is then given by the sum over all cells
\begin{equation}
\E^{\cb}(\by) = \sum_{j=-N}^N \E_j^{\cb}(\by) = \frac{1}{2}
\sum_{j=-N}^N\int_{Q_j}\rho_{\by}\psi^{(j)}\dx.
\label{eq:CBenergy}
\end{equation}
In the Cauchy--Born model we seek to minimize the total potential
energy $E_{\bs{f}}^{\cb}:\mc{Y}\rightarrow \R$ defined by
\begin{equation}
  \label{eq:totalenergymin_cb}
  E_{\bs{f}}^{\cb}(\by) = \E^{\cb} (\by) +
  (\bs{f},\by)_{\veps}.
\end{equation}
Whether the Cauchy--Born model is a good approximation to the exact
atomistic model strongly depends on the regularity properties of
minimizers of \eqref{eq:totalenergymin_cb}.


Let $\bu\in\mc{U}$ be a test vector and $u\in\calS_\#(\by)$ an
interpolant of $\bu$, i.e., $u(y_j) = u_j$ for $j \in \Z$. It follows
as in Lemma \ref{lemma:DEyN} that the derivative of
$\E_j^{\cb}(\by)$ can be written in the form
\begin{equation}
  \label{eq:DEcbj_weak}
  D_{\by}\E^{\cb}_j(\by)\Didot\bu = \frac{u_j-u_{j-1}}{y_j-y_{j-1}}\int_{Q_j}
  \sigma^{\cb}_{j,\by}(x)\dx = \int_{Q_j} \sigma^{\cb}_{j,\by}(x)\nabla u(x)\dx,
\end{equation}
where the local continuum stress function $\sigma^{\cb}_{j,\by}$,
in direct correspondence with \eqref{eq:sigmaat12}, is
\begin{equation}
\begin{split}
\sigma^{\mrmc}_{j,\by}(x) =~&
\tfrac{1}{2}\onept\veps^2|\nabla\psi^{(j)}(x)|^2-\tfrac{1}{2}\onept
m^2\psi^{(j)}(x)^2+\rho_\by(x)\psi^{(j)} (x)
\\ &~+ \veps\sum_{j=-N-1}^N\psi^{(j)}(x) \nabla\delta_\varepsilon(x-y_j)(x-y_j).
\end{split}
\label{eq:contstress}
\end{equation}
Furthermore, we define the Cauchy--Born stress function
$\sigma_{\by}^{\cb}:\Omega\rightarrow\R$ by
\begin{equation*}
 \sigma_{\by}^{\cb}(x) = \sigma^{\mrmc}_{j,\by}(x) \quad \text{if}\ \ x\in\Omega_{j}
\end{equation*}
for all $x\in\Omega$.


\subsection{Consistency}
Next, we turn to the consistency analysis of the Cauchy--Born
approximation, for which we thoroughly analyze the modelling error
incurred. From \eqref{eq:per_weakform} and \eqref{eq:DEcbj_weak} we
deduce that
\begin{equation*}
\begin{split}
\bigl|\bigl(D\E(\by)-D\E^{\cb}(\by)\bigr)\Didot\bu\bigr| \leq&
\int_{\Omega}\bigl|\sigmaa_{\by}(x)-\sigma^{\mrmc}_{\by}(x)\bigr|\onept |\nabla
u(x)|\dx\\
=&\sum_{j=-N}^N \int_{Q_j} \bigl|\sigmaa_{\by}(x)-\sigma^{\mrmc}_{j,\by}
(x)\bigr|\onept|\nabla u(x)|\dx,
\end{split}
\end{equation*}
where the stress functions $\sigmaa_{\by}$ and
$\sigma^{\mrmc}_{j,\by}$ are given by \eqref{eq:sigmaat12} and
\eqref{eq:contstress}, respectively. 
To investigate the modelling error
$\bigl|\sigmaa_{\by}(x)-\sigma^{\cb}_{j,\by}(x)\bigr|$ incurred by
going from the atomistic description to the Cauchy--Born approximation
it is therefore sufficient to analyze $|\phi-\psi^{(j)}|$ and
$|\nabla\phi-\nabla\psi^{(j)}|$ in $Q_j$ for every
$j\in\{-N,\ldots,N\}$.


\begin{lemma}
  \label{lemma:ypp}
  Let $\by\in\ell^\infty(\Z)$ and define
  $\by^{(j)}=(y^{(j)}_k)_{k\in\Z}$ by $y^{(j)}_k = y_j + \varepsilon
  y_j'(k-j)$ for all $k\in\Z$; then
  \begin{equation*}
    \begin{split}
      \bigl|y_n-y_n^{(j)}\bigr| \leq~& (n-j) \veps^2
      \norm{\by''}_{\ell^1([j,n-1])} \qquad \text{for } n > j, \\
      \bigl|y_n-y_n^{(j)}\bigr| \leq~&
      (j-1-n)\veps^2\norm{\by''}_{\ell^1([n+1,j-1])}, \qquad
      \text{for } n < j-1.
    \end{split}
  \end{equation*}
\end{lemma}

\begin{proof}
Assume, without loss of generality that $n > j$. Since $y_{j-1} = y^{(j)}_{j-1}$ and $y_{j} = y^{(j)}_{j}$,
\begin{equation*}
 y_n-y_n^{(j)} =\varepsilon \sum_{k=j+1}^n \bigl(y_k'-(y_k^{(j)})'\bigr) = \varepsilon^2
\sum_{k=j+1}^n\sum_{l=j}^{k-1} \bigl(y_l''-(y_l^{(j)})''\bigr)= \varepsilon^2
\sum_{k=j+1}^n\sum_{l=j}^{k-1} y_l'' ,
\end{equation*}
where have used that $(\by^{(j)})'$ is constant. Changing the order of
summation we get
\begin{equation*}
  |y_n-y_n^{(j)}| \leq \varepsilon^2 \sum_{l=j}^{n-1}\sum_{k=l+1}^{n} |y_l''| = \varepsilon^2
\sum_{l=j}^{n-1}(n-l) y_l'' \leq (n-j) \veps^2
      \norm{\by''}_{\ell^1([j,n-1])}. \qedhere
\end{equation*}
\end{proof}

In the next result we estimate the errors $|\phi(x)-\psi^{(j)}(x)|$,
$|\nabla\phi(x)-\nabla\psi^{(j)}(x)|$ for $x$ in the cell $Q_j$.  As
anticipated by Lemma \ref{lemma:ypp} they depend on the second
difference $\by''$.

\begin{lemma}
\label{lemma:psijconvergence}
Let $\by\in\mc{Y}$ satisfy $\min\by'>\vsig$. Let $\phi\in\mrm{H}^1_\#(\Omega)$ satisfy
\eqref{eq:phiPDE}
and $\psi^{(j)}\in\mrm{H}^1_\#(Q_j)$ satisfy \eqref{eq:psijPDE}, respectively. Then,
\begin{align*}
\bigl\lVert \phi-\psi^{(j)} \bigr\rVert_{\mrm{L}^\infty(Q_j)} &\leq
\ceff\veps\sum_{n=1}^\infty
\norm{\by''}_{\ell^1([j-n,j+n-1])}n \e^{-m n \min \by'},\quad\text{and}\\
\bigl\lVert \veps\nabla\phi-\veps\nabla\psi^{(j)}\bigr\rVert_{\mrm{L}^\infty(Q_j)} &\leq
m\ceff\veps\sum_{n=1}^\infty
\norm{\by''}_{\ell^1([j-n,j+n-1])}n \e^{-m n \min \by'}.
\end{align*}
\end{lemma}

\begin{proof}
  From Proposition \ref{prop:phiexpression} we immediately deduce
  that, for all $x\in Q_j$,
\begin{equation}
\begin{split}
 \phi(x) =~& \frac{1}{2m}\int_\R  \sum_{k\in\Z} \delta_\veps(z-y_k)
\e^{-\tfrac{m}{\veps}|x-z|}\dz,\\
 \psi^{(j)}(x) =~& \frac{1}{2m}\int_\R \sum_{k\in\Z} \delta_\veps(z-y_k^{(j)})
\e^{-\tfrac{m}{\veps}|x-z|}\dz.
\end{split}
\label{eq:phiandpsij}
\end{equation}
Since $y^{(j)}_j = y_j$ and $y^{(j)}_{j-1} = y_{j-1}$, the respective terms in
the sums cancel. Hence, we get for $x\in Q_j$:
\begin{equation*}
\phi(x)-\psi^{(j)}(x)= \frac{1}{2m}\sum_{\tworow{k\in\Z}{k\neq j-1,j}}
\int_\R \bigl(\delta_\veps(z-y_k) -
\delta_\veps(z-y^{(j)}_k) \bigr) \e^{-\tfrac{m}{\veps}|x-z|}\dz.
\end{equation*}
We now derive bounds on the individual terms in the sum. Note that
\eqref{eq:mueffcalculation} simplifies the following calculations but
due to the smoothness of the Green's function similar bounds can be
obtained without it.  

Let $k>j$. Then we have $|x-z|=z-x$ for all $z\in\supp
\delta_\veps(\cdot-y_k)$ and all
$z\in\supp\delta_\veps(\cdot-y_k^{(j)})$. Thus, with
\eqref{eq:mueffcalculation},
\begin{equation}
\frac{1}{2m}\int_\R \bigl(\delta_\veps(z-y_k) -
\delta_\veps(z-y^{(j)}_k) \bigr) \e^{-\tfrac{m}{\veps}|x-z|}\dz =
\frac{\ceff}{2m}\bigl(\e^{-\tfrac{m}{\veps}(y_k-x)}-\e^{-\tfrac{m}{\veps}(y_k^{(j)}-x)}\bigr).
\label{eq:phiminuspsi}
\end{equation}
If $y_k^{(j)}\geq y_k$, then
\begin{equation*}
\begin{split}
\biggl|\frac{1}{2m}\int_\R \bigl(\delta_\veps(z-y_k) -
\delta_\veps(z-y^{(j)}_k) \bigr) \e^{-\tfrac{m}{\veps}|x-z|}\dz\biggr| \leq~&
\frac{\ceff}{2m}\e^{-\tfrac{m}{\veps}(y_k-x)}\bigl(1-\e^{-\tfrac{m}{\veps}(y_k^{(j)}-y_k)}\bigr)\\
\leq~&\frac{\ceff}{2m}\e^{-\tfrac{m}{\veps}(y_k-x)}\frac{m}{\veps}(y_k^{(j)}-y_k).
\end{split}
\end{equation*}
Using $(y_k-x)\geq (k-j)\veps\min\by'$ for all $x\in Q_j$ and applying Lemma \ref{lemma:ypp} leads
to
\begin{equation*}
 \frac{\ceff}{2\veps}\e^{-\tfrac{m}{\veps}(y_k-x)}\bigl|y_k-y_k^{(j)}\bigr|
 \leq
 \frac{\ceff\veps}{2} \norm{\by''}_{\ell^1([j,k-1])}(k-j)\e^{-(k-j)m\min \by'}.
\end{equation*}
The same bound on \eqref{eq:phiminuspsi} can be obtained if $y_k^{(j)}\leq y_k$. 

For any $k<j-1$ we can use the same techniques to obtain that
 \begin{equation*}  
\begin{split}
\biggl|\frac{1}{2m}\int_\R \bigl(\delta_\veps(z-y_k) -
\delta_\veps(z-y^{(j)}_k) \bigr)&~
\e^{-\tfrac{m}{\veps}|x-z|}\dz\biggr|\\ & 
\hspace{-2cm} \leq
\frac{\ceff\veps}{2} \norm{\by''}_{\ell^1([k+1,j-1])}(j-k-1)\e^{-(j-k-1)m\min \by'}.
\end{split}
\end{equation*}

Summing over all $k\in\Z\backslash\{j-1,j\}$ we deduce that
\begin{equation*}
  \bigl|\phi(x)-\psi^{(j)}(x)\bigr| \leq \ceff\veps\sum_{n=1}^\infty
\norm{\by''}_{\ell^1([j-n,j+n-1])}n \e^{-m n \min \by'}.
\end{equation*}

The proof for the derivatives $\nabla\phi$, $\nabla\psi^{(j)}$ is analogous.
\end{proof}

We wish to prove modelling error estimates on $\bigl\lVert\sigmaa_\by
- \sigma^{\cb}_{j,\by}\bigr\rVert_{\mrm{L}^\infty(Q_j)}$ in terms
of $\bigl\lVert \phi-\psi^{(j)}\bigr\rVert_{\mrm{L}^\infty(Q_j)}$ and
$\bigl\lVert \nabla\phi - \nabla\psi^{(j)}
\bigr\rVert_{\mrm{L}^\infty(Q_j)}$. Since the stress functions
$\sigma_{\by}$ and $\sigma^{\cb}_{j,\by}$ are quadratic in the
fields $\phi$ and $\psi^{(j)}$ we need $\mrm{L}^\infty$-bounds, which
we establish in the next lemma.

\begin{lemma}
  \label{lemma:phiLinftynorm}
  Let $\by\in\mc{Y}$, $\by'>\vsig$, and let
  $\phi=\arg\min_{\varphi\in\mrm{H}^1_\#(\Omega)} I(\varphi,\by)$ be
  the corresponding field. Then, there are continuous functions $K_0$,
  $K_1$, that depend implicitly on $m$ (but are independent of $\veps$
  and $\by$), such that
  \begin{equation*}
    \lVert \phi\rVert_{\mrm{L}^\infty(\Omega)} \leq K_0(m\min\by'),
    \quad \text{and} \quad 
    \veps\lVert \nabla\phi\rVert_{\mrm{L}^\infty(\Omega)} \leq K_1(m\min\by').
  \end{equation*}
\end{lemma}
\begin{proof}
  The stated estimates follow in a straightforward manner from the
  integral representation of the solution $\phi$; see \cite[Lemma
  4.4]{BLthesis} for the details.
\end{proof}

We can now prove the following modelling error estimates.

\begin{lemma}
Let $\sigmaa_\by$ and $\sigma^{\cb}_{j,\by}$ be given by \eqref{eq:sigmaat12},
respectively, \eqref{eq:contstress}; then
\begin{equation*}
\begin{split}
 \bigl\lVert\sigmaa_\by - \sigma^{\cb}_{j,\by}\bigr\rVert_{\mrm{L}^\infty(Q_j)}
\leq~&  C \big( \veps\,\bigl\lVert
\nabla\phi-\nabla\psi^{(j)}\bigr\rVert_{\mrm{L}^\infty(Q_j)}
 + \bigl\lVert
\phi-\psi^{(j)}\bigr\rVert_{\mrm{L}^\infty(Q_j)}\big), \quad j = 1,
\dots, N, 
\end{split}
\end{equation*}
where the constant $C$ only depends on $\delta_1$, $K_i = K_i(m\by')$,
and on $m$.
\label{lemma:contstressconsistency}
\end{lemma}

\begin{proof}
From the definitions of the atomistic and continuum stress function we deduce that
\begin{equation*}
\begin{split}
\sigmaa_{\by}(x)-\sigma^{\mrmc}_{j,\by}(x) =~&
-\tfrac{1}{2}\bigl(\veps\nabla\phi(x) -\veps\nabla\psi^{(j)}(x)\bigr)\bigl(\veps\nabla\phi(x)
+\veps\nabla\psi^{(j)}(x)\bigr)\\ 
&+ \tfrac{1}{2}m^2\bigl(\phi(x)-\psi^{(j)}(x)\bigr)\bigl(\phi(x)+\psi^{(j)}(x)\bigr)
-\rho_\by(x)\bigl(\phi(x) - \psi^{(j)}(x)\bigr) \\
&-\bigl(\phi(x) - \psi^{(j)}(x)\bigr)\sum_{i=j-1}^j  \veps\nabla\delta_\varepsilon(x-y_i)(x-y_i)
\end{split}
\end{equation*}
for all $x\in Q_j$. With $\delta_\veps(x) =
\veps^{-1}\delta_1(x/\veps)$, the $\mrm{L}^\infty$-bound on $\phi$
from Lemma \ref{lemma:phiLinftynorm}, and the analogous bound for
$\psi^{(j)}$ we get
\begin{equation*}
\begin{split}
 \tfrac{1}{2} \bigl|\veps\nabla\phi(x)
+\veps\nabla\psi^{(j)}(x)\bigr|\leq~& K_1(m\min\by'),\\
 \tfrac{m^2}{2}\bigl|\phi(x) +\psi^{(j)}(x)\bigr| \leq~& m^2K_0(m\min\by'),\\
\norm{\rho_\by}_{\mrm{L}^\infty} \leq~&\norm{\delta_1}_{\mrm{L}^\infty},\\
\bigl|\veps\nabla\delta_\varepsilon(x-y_{i})(x-y_{i})\bigr|
\leq~&\lVert\nabla  \delta_1 {\rm id}\rVert_{\mrm{L}^\infty},
\end{split}
\end{equation*}
which implies the stated result.
\end{proof}

\subsection{Stability}
Besides consistency, the second crucial property of an approximation
to a given model is its stability. The following auxiliary result will
play a role in the stability analysis of a/c methods.

\begin{lemma}
  Let $\by\in\mc{Y}$ satisfy $\min\by'>\vsig$. Then, for all
  $j\in\{-N,\ldots,N\}$,
\begin{equation*}
 D^2\E^{\cb}_j(\by)\Didot[\bu,\bu] \geq \frac{m^2\ceff^2}{2}\onept\e^{- m\max \by'}\veps
|u_j'|^2\quad \forall \bu\in\mc{U}.
\end{equation*}
\label{lemma:cbstability}
\end{lemma}

\begin{proof}
We first recall that $\E^{\cb}_j(\by) = \tfrac{1}{2}\int_{Q_j}\rho_\by\psi^{(j)}\dx$ because
$\psi^{(j)}$ is a minimizer of \eqref{eq:Ecb}. Extending $\psi^{(j)}$ $|Q_j|$-periodically to $\R$
and using the symmetry of the cell problem, we can rewrite this as
\begin{equation*}
 \E^{\cb}_j(\by) = \frac{\veps}{2} \int_{\R}\delta_\veps(x-y_j)\psi^{(j)}(x)\dx.
\end{equation*}
We now insert the explicit formula \eqref{eq:phiandpsij} for $\psi^{(j)}(x)$ and apply
\eqref{eq:mueffcalculation} to get
\begin{equation*}
\begin{split}
 \E_j^{\cb}(\by)=~&\frac{\veps}{4m} \sum_{k\in\Z} \int_\R\int_\R  \delta_\veps(x-y_j)
\delta_\veps(z-y_k^{(j)}) \e^{-\tfrac{m}{\veps}|x-z|}\dz \dx \\
=~&\frac{\ceff^2\veps}{4m} \sum_{\tworow{k\in\Z}{k\neq j}}
\e^{-\tfrac{m}{\veps}|y_j-y^{(j)}_k|} + \E_{{\rm self}}
=\frac{\ceff^2\veps}{2m} \sum_{\nu=1}^\infty
\e^{-m\onept \nu\onept y'_j} + \E_{{\rm self}},
\end{split}
\end{equation*}
where the constant $\E_{{\rm self}}$ coming from $k=j$ in the sum represents the self-energies of
the atoms in the cell $Q_j$. Here we have also used that $\bigl|y_k^{(j)}-y_j\bigr|= |k-j| y_j'$ for
all $k\in\Z$. Differentiating twice leads to
\begin{equation*}
\begin{split}
  D^2\E^{\cb}_j(\by)\Didot[\bu,\bu] =~& \frac{m\ceff^2}{2}\onept\veps \sum_{\nu=1}^\infty \nu^2
\e^{-\nu my'_j}|u_j'|^2\\
\geq ~& \frac{m\ceff^2}{2}\onept\veps |u_j'|^2\sum_{\nu=1}^\infty \nu^2
\e^{-\nu m\max \by'}
\geq\frac{m\ceff^2}{2}\onept \e^{- m\max \by'}\ \veps |u_j'|^2.
\end{split}
\end{equation*}
In the last step we have only kept the term for $\nu=1$, which
represents the nearest neighbour interactions.
\end{proof}

\section{Atomistic-to-Continuum Coupling}
\label{Sec:QCCoupling}
The computation of the original atomistic energy $\E(\by)$ involves the solution of the optimization
problem \eqref{eq:Eatomistic} posed in the whole of $\Omega = (y_{-N-1},y_N)$. Our goal is the
construction of computationally cheaper, approximate energies $\E^{\qc}(\by)$ such that
$\E(\by)\approx\E^{\qc}(\by)$ for all relevant $\by$ and minimizers
$\bar{\by}^{\qc}\in\mc{Y}$ of
\begin{equation*}
 E^{\qc}_{\bs{f}}(\by) = \E^{\qc}(\by) + (\bs{f},\by)_{\veps},
\end{equation*}
are good approximations of minimizers $\bar{\by}$ of the energy $E_{\bs{f}}$ from
\eqref{eq:totalenergymin}.

Following the philosophy of a/c methods we approximate $\E(\by)$ by
the continuum model where $\by$ is smooth and a version of the
atomistic model where $\by$ is nonsmooth.  In the following we will
implicitly assume that the configurations $\by\in\mc{Y}$ under
consideration are smooth except in the segment $y_{-K},\ldots,y_K$ for
some $K<N$. We divide $\Omega$ into an atomistic subdomain
$\Omega^{\mrma}$ such that $y_j\in\Omega^{\mrma}$ for all
$j\in\{-K,\ldots,K\}$ and the continuum domain
$\Omega^{\cb}=\Omega\backslash\Omega^{\mrma}$.  In
$\Omega^{\cb}$ we will use the Cauchy--Born approximation on a
cell-by-cell basis. In $\Omega^{\mrma}$ we will use the atomistic
model with Dirichlet boundary conditions as discussed in Section
\ref{sec:SMDirBCs}.

This basic setting gives rise to a variety of possibilities including
the precise choice of $\partial\Omega^{\mrma}$ and the boundary
conditions imposed on the atomistic subproblem.  Both will in general
depend on the configuration $\by$. Our main objective for
$\E^{\qc}$ is the existence of a weak formulation in the sense
that
\begin{equation*}
 D\E^{\qc}(\by)\Didot\bu = \intO \sigma^{\qc}_{\by}(x)\nabla u(x)\dx,
\end{equation*}
where $u\in\calS_\#(\by)$ is a piecewise linear interpolant of
$\bu\in\mc{U}$ and $\sigma^{\qc}_{\by}$ is a stress function to
be determined. If this weak formulation can be obtained, the
consistency analysis reduces to error estimates on fields, as already
seen in Lemma \ref{lemma:contstressconsistency}.

Throughout this section, $\phi\in\mrm{H}^1_\#(\Omega)$ denotes the
solution of the original minimization problem \eqref{eq:Eatomistic}
for a given configuration $\by\in\mc{Y}$.

\subsection{An a/c method with optimal boundary conditions}
\label{sec:firstmethod}
\begin{figure}
\begin{center}
\includegraphics[width=.7\linewidth]{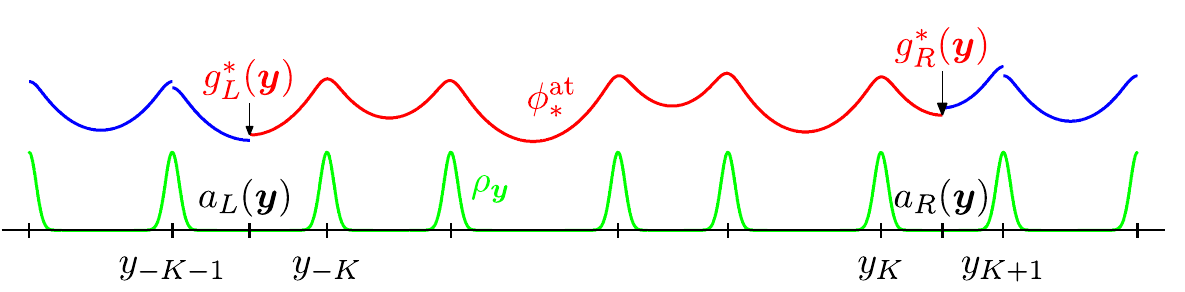}
\caption{An illustration of the first a/c method. In $\Omega^{\mrma}=(\aL(\by),\aR(\by))$ the
atomistic problem is solved with the Dirichlet boundary conditions $g^*(\by)$. Outside
$\Omega^{\mrma}$ the Cauchy--Born approximation is used in all cells $Q_j$.
}\label{fig:QC_Method1}
\end{center}
\end{figure}
We place the boundary $\aLR$ of the atomistic
subproblem halfway between the interface atoms, that is $\aLR =
\aLR(\by)=[\aL(\by)\, \, \aR(\by)]^{\rm T}$, where
\begin{equation*}
 \aL(\by) = \frac{y_{-K-1}+y_{-K}}{2}, \qquad \aR(\by) = \frac{y_{K}+y_{K+1}}{2}.
\end{equation*}
Let $\Omega^{\mrma}=(\aL(\by),\aR(\by))$ and $\Omega^{\mrmc} =
\Omega\backslash\Omega^{\mrma}$. We write the a/c energy $\Eqc(\by)$ as
the sum of a continuum and an atomistic part
\begin{equation}
 \Eqc(\by) = \Ec_*(\by) + \Ea_*(\by), \label{eq:EQC_Method1}
\end{equation}
which are introduced below.

Due to the choice of $\aLR(\by)$ there are two half cells,
$(y_{-K-1},\aL(\by))$ and $(\aR(\by),y_{K+1})$, in the continuum
region $\Omega^{\mrmc}$ (see Figure \ref{fig:QC_Method1}). Since the
cell problems are symmetric, the Cauchy--Born energies of these half
cells are given by $\tfrac{1}{2}\Ec_{-K}(\by)$ and
$\tfrac{1}{2}\Ec_{K+1}(\by)$, respectively. Hence, the continuum
contribution to the energy $\Eqc$ is defined by
\begin{equation}
\Ec_*(\by) = \sum_{j=-N+1}^{-K-1} \Ec_j(\by)+\tfrac{1}{2}\Ec_{-K}(\by)
+ \tfrac{1}{2}\Ec_{K+1}(\by)+ \sum_{j=K+2}^{N} \Ec_j(\by).
\label{eq:Ecbdefinition}
\end{equation}

The coordinates of the atoms in the atomistic region $\Omega^{\mrma}$
are represented by
\begin{equation*}
\bya = (y_{-K},\,\ldots\,,y_K)^{\rm T}.
\end{equation*}
For the definition of $\E^{\mrma}_*(\by)$ we consider the minimization
problem \eqref{eq:Ebdd} on the atomistic domain $\Omega^{\mrma}$
subject to the Dirichlet boundary conditions $g^*(\by)=[g_L^*(\by)\,\,
g_R^*(\by)]^{\rm T}$. In correspondence with Remark \ref{remark:refl} and
Section \ref{sec:specialcase} they are given by
\begin{equation*}
g_{L}^*(\by) = \frac{1}{1-\tau}\frac{\gamma_{L}(\by)+\tau \gamma_{R}(\by)}{1+\tau},\qquad 
g_{R}^*(\by) = \frac{1}{1-\tau}\frac{\tau \gamma_{L}(\by)+\gamma_{R}(\by)}{1+\tau},
\end{equation*}
where $\tau=\e^{-\tfrac{m}{\veps}\Delta a(\by)}$, and $\gamma_L,
\gamma_R$ are defined in \eqref{eq:gammaLgammaR}.  The energy
contribution from the atomistic subproblem is thus given by
\begin{equation*}
\begin{split}
\E^{\mrma}_*(\by) =~& \E_{\aLR(\by),g^*(\by)}(\bya)
= -\inf \biggl\{ I_{\aLR(\by)}(\varphi,\bya):\ \ \varphi\in\mrm{H}^1(\Omega^{\mrma}),\quad
\varphi|_{\partial\Omega^{\mrma}}=g^*(\by) \biggr\},
\end{split}
\end{equation*}
where $I_{\aLR(\by)}$ is defined as in \eqref{eq:I_a}. We denote the
solution of this optimization problem by
$\phi_{\mrma}^*\in\mrm{H}^1(\Omega^{\mrma})$. It satisfies the
boundary-value problem
\begin{equation*}
\begin{split}
 -\veps^2\Delta \phi_{\mrma}^* + m^2\phi_{\mrma}^* =~& \rho_\by\ \ \text{in } \Omega^\mrma,\\[1mm]
\phi_{\mrma}^*|_{\partial\Omega^{\mrma}} =~& g^*(\by).
\end{split}
\end{equation*}
From a computational point of view $g^*(\by)$ is also a convenient
choice since this is equivalent to homogeneous Neumann boundary
conditions. In Section \ref{sec:specialcase} we deduced a clear
interpretation of the effect of this choice of boundary data: besides
the interaction among themselves, the atoms in $\Omega^{\mrma}$
interact with mirror atoms outside $\Omega^{\mrma}$. This is closely
related to the geometric reconstruction idea for classical potentials
\cite{Shimokawa, ELuYang_Uniform}.

In analogy to \eqref{eq:totalenergymin} we search for minimizers of
the total potential energy
\begin{equation}
 E^{\qc}_{\bs{f}}(\by) = \E^{\qc}(\by) + (\bs{f},\by)_{\veps}
\label{eq:totalenergyminQC}
\end{equation}
in $\mc{Y}$, where $\bs{f}\in\mc{U}^{-1,2}$ represents an external
force. Formally, a minimizer $\bar\by^{\qc}$ satisfies the following
Euler--Lagrange equation in $\mc{U}^{-1,2}$:
\begin{equation*}
 DE^{\qc}_{\bs{f}}(\by) = D\E^{\qc}(\by) + \bs{f} = {\bf 0}.
\end{equation*}
Throughout the remainder of this article we assume that the atomistic
domain $\Omega^{\mrma}$ is large compared with $\veps$, that is
$\Delta a \gg \veps$ and hence terms of order $\mc{O}(\tau)$ are
exponentially small.
To keep the notation more compact we will not give precise estimates
of $\tau$-dependent terms arising from the atomistic domain explicitly
but include an $\mc{O}(\tau)$ where necessary.

\subsection{Consistency}
In order to study the consistency properties of the a/c energy
$\Eqc(\by)$ from \eqref{eq:EQC_Method1} we first need to calculate its
derivative. Having established weak formulations for the derivatives
of $\E$, $\E^{\mrmc}$, as well as $\E_{\aLR,g}$, we will prove that
the a/c energy $\E^{\qc}$ admits a similar
reformulation of $D\E^{\qc}(\by)\cdot\bu$. For this we have to
take into account that both the boundary of the atomistic domain
$\Omega^{\mrma}$ and the boundary conditions depend on $\by$. The
necessary preparations were carried out in Section \ref{sec:SMDirBCs}.

\begin{lemma}
  Let $\by\in\mc{Y}$ satisfy $\min\by'>\vsig$. Furthermore, let $\bu
  \in \mathcal{U}$ be a test vector and $u\in \Sm_{\#}(\by)$ an
  interpolant of $\bu$; then,
  \begin{equation}
    D\E^{\qc}(\by)\Didot\bu = \int_\Omega \sigma^{\qc}_\by(x)\nabla
    u(x)\dx,
    \qquad \text{where } \sigma_\by^{\qc}(x) = \left\{\begin{array}{ll}
        \sigma_\by^{\cb}(x)& \text{if }x\in\Omega^{\mrmc},\\[1mm]
        \sigma_{\by,*}^{\mrma}(x)& \text{if }x\in\Omega^{\mrma},
      \end{array}\right. 
    \label{eq:QCweakform1}
  \end{equation}
  and $\sigma_{\by,*}^{\mrma}(x)$ is given by \eqref{eq:sigmaat12}
  with $\phi=\phi_\mrma^*$.
  \label{lemma:Method0WeakForm}
\end{lemma}

\begin{proof}
  \emph{1. Continuum Contribution.} From Section \ref{sec:SM_CB} we
  already have the equality
\begin{equation*}
 D\Ec_j(\by)\Didot\bu = \int_{Q_j} \sigma^{\mrmc}_{\by,j}(x)\nabla u(x)\dx,
\end{equation*}
$j\in\{ -N,\ldots,-K-1\}\cup\{K+2,\ldots,N\}$. For the contribution $\tfrac{1}{2}\Ec_{-K}(\by)$ from
the half cell $(y_{-K-1},\aL(\by))$ we make use of the symmetry of the cell problems. Since $\nabla
u|_{Q_{-K}}$ is constant, $\aL(\by)$ is the midpoint of $Q_{-K}=(y_{-K-1},y_{-K})$, and
$\sigma^{\mrmc}_{\by,-K}$ is symmetric in $Q_{-K}$, we deduce that
\begin{equation*}
 \tfrac{1}{2} D\Ec_{-K}(\by)\Didot\bu =\frac{1}{2} \int_{Q_{-K}}
\sigma^{\mrmc}_{\by,-K}(x)\nabla u(x)\dx =
\int_{y_{-K-1}}^{\aL(\by)}\sigma^{\mrmc}_{\by,-K}(x)\nabla u(x)\dx.
\end{equation*}
We treat $\tfrac{1}{2}\Ec_{K+1}(\by)$ analogously. Hence,  
\begin{equation*}
  D\Ec_*(\by)\Didot\bu = \int_{\Omega^{\mrmc}} \sigma^{\mrmc}_{\by}(x)\nabla u(x)\dx,
\end{equation*}
where $\sigma_{\by}^{\mrmc}(x) = \sigma^{\mrmc}_{\by,j}(x)$ if $x\in Q_j$. 

\emph{2. Atomistic Contribution.} To calculate the derivative $D\E^{\mrma}_{*}(\by)$ we use
the chain rule and the derivatives that were provided in Section \ref{sec:SMDirBCs}. Applying
Proposition \ref{cor:Ebdd_weakform} (with $h_L = (u_{-K-1}+u_{-K})/2$, $h_R = (u_{K}+u_{K+1})/2$
because of $D_\by a(\by)\cdot \bu = a(\bu)$), we get
\begin{equation}
\begin{split}
D\E^{\mrma}_{*}(\by)\Didot\bu=~&D_{\by}\E_{\aLR(\by),g^*(\by)}(\bya)\Didot \bu_{\mrma} +
D_{\aLR}\E_{\aLR(\by),g^*(\by)}(\bya)\Didot D_{\by}\aLR(\by)\Didot\bu\\
=~& \int_{\Omega^{\mrma}}\sigma^{\mrma}_{\by,*}(x)\nabla u(x)\dx,
\end{split}
\end{equation}
where the stress $\sigma^{\mrma}_{\by,*}$ is given by \eqref{eq:sigmaat12} with
$\phi=\phi_{\mrma}^*$ and $\bu_{\mrma} = (u_{-K},\ldots,u_K)\in\R^{2K+1}$ is the section of $\bu$
corresponding to the atoms in the atomistic region. Note that the choice of boundary conditions
implies $D_{g}\E_{\aLR(\by),g^*(\by)}(\bya)=0$; cf. Remark \ref{remark:refl}.
\end{proof}

\begin{remark}
  The weak form \eqref{eq:QCweakform1} of the derivative
  $D\E^{\qc}$ already implies that there are no ghost forces for
  homogeneous deformations $\by$. If the atoms are equidistant, then
  $g^*_L(\by) = \phi(\aL)$ and $g^*_R(\by) = \phi(\aR)$ and thus also
  $\phi^*_\mrma = \phi$ in $\Omega^\mrma$. Moreover, it is clear that
  $\psi^{(j)} = \phi$ for all $j$.  Hence, we obtain that
  $\sigma^{\qc}_{\by}(x) = \sigmaa_\by(x)$ for all $x\in\Omega$,
  which implies that $D\E^{\qc}(\by) = D\E(\by) = 0$ for all
  $\by=F\hat{\bs{X}}\in\mc{Y}$ representing homogeneous deformations
  (i.e., that the method exhibits no ghost forces).
\end{remark}

\medskip

Absence of ghost forces does not immediately imply {\it consistency}
of the a/c method, but has to be shown separately. This we do
next. Because of the structure of the weak formulation
\eqref{eq:QCweakform1}, the analysis boils down to estimating the
errors between the field $\phi$ coming from the original atomistic
model and the fields $\psi^{(j)}$, respectively, $\phi^*_{\mrma}$.

\begin{theorem}
  \label{Lemma:Method0Consistency}
  Let $\by\in\mc{Y}$ be such that $\min\by'\geq s_0>\vsig$; then, for all
  $\bu \in \mc{U}$ with interpolants $u\in\calS_\#(\by)$,
  \begin{equation*}
    \bigl|\bigl(D\E(\by)-D\E^{\qc}(\by)\bigr)\Didot\bu\bigr| \leq
    C \big( \veps \norm{ \by''}_{\ell^2_{w,s_0}} + \tau \big)\, \norm{\nabla
      u}_{\mrm{L}^2},
  \end{equation*}
  where $C = C(s_0)$ and the weighted $\ell^2_{w,s_0}$-norm is defined
  by
  \begin{align}
    \label{eq:defn_ell2_weighted}
    \norm{ \by''}_{\ell^2_{w,s_0}}^2 :=~& \veps {\textstyle \sum_{j = -N}^N} w_j
    \big|y_j''\big|^2,
  \end{align}
  with weights $w_j := \max\big(1, \e^{-ms_0 {\rm dist}(j, \{-K,
    K\})}\big)$.
\end{theorem}
\begin{proof}
  Using the weak formulation \eqref{eq:QCweakform1} of $D\Eqc(\by)$ we obtain
  \begin{align}
    \notag
    \bigl|\bigl(D_{\by}\E(\by)-D_{\by}\Eqc(\by)\bigr)\Didot\bu\bigr|=~& \left|\intO
      \bigl(\sigmaa_{\by}(x)-\sigma^{\mathrm{qc}}_{\by}(x)\bigr)\nabla
      u(x)\dx\right| \\ 
    \notag
    \leq
    ~&\norm{\sigmaa_{\by}-\sigma^{\mathrm{qc}}_{\by}}_{{L}^2(\Omega)}\HneN{u}\\
    \label{eq:DEqcminusDE} 
    \leq ~& \Big( \sum_{j = -N}^N \veps \|
    \sigmaa_{\by}-\sigma^{\mathrm{qc}}_{\by} \|_{L^\infty(Q_j)}^2
    \Big)^{1/2} \cdot \HneN{u}.
 \end{align}
  For $Q_j$ belonging to the continuum region Lemma
  \ref{lemma:contstressconsistency} and Lemma
  \ref{lemma:psijconvergence} imply
  \begin{align}
    \notag
    \| \sigmaa_{\by}-\sigma^{\mathrm{qc}}_{\by} \|_{L^\infty(Q_j)}
    \leq~& C \veps \sum_{n = 1}^\infty \| \by'' \|_{\ell^1([j-n, j+n-1])}
    n \e^{-mns_0} \\
    \notag
   \leq~& C \veps \sum_{n = 1}^{\infty} \| \by'' \|_{\ell^2([j-n, j+n-1])}
    n^{3/2} \e^{-mns_0} \\
    \label{eq:cons_err_contregion}
    \leq~& C \veps \bigg( \sum_{n = 1}^\infty \| \by'' \|_{\ell^2([j-n,
      j+n-1])}^2 \e^{-m n s_0} \bigg)^{1/2},
  \end{align}
  where we have employed the Cauchy--Schwarz inequality twice and
  used the fact that the series $\sum_{n = 1}^\infty n^3 \e^{-mns_0}$
  is convergent.

  Summing over all cells belonging to the continuum region and
  interchanging the order of summation we obtain
  \begin{align*}
    \sum_{\substack{j \in \{-N, \dots, N\} \\ \setminus\{-K+1,
        \dots, K\}}} \hspace{-4mm} \veps \|
    \sigmaa_{\by}-\sigma^{\mathrm{qc}}_{\by} \|_{L^\infty(Q_j)}^2
    \leq~& C \veps^3 \hspace{-4mm} \sum_{\substack{j \in \{-N, \dots, N\} \\ \setminus\{-K+1,
        \dots, K\}}} 
    \sum_{n = 1}^{\infty} \| \by'' \|_{\ell^2([j-n, j+n-1])}^2
    \e^{-mns_0} \\
    \leq~& C \veps^3 \sum_{k = -N}^N w_k' |y_k''|^2, 
  \end{align*}
  where
  \begin{displaymath}
    w_k' = \sum_{\substack{j \in \{-N, \dots, N\} \\ \setminus\{-K+1,
        \dots, K\}}} \sum_{\substack{n = 1, \dots, \infty \\ k \in [j-n,
          j+n-1]}} \e^{-mns_0}.
  \end{displaymath}
  This is a geometric series from which we can factor out $\e^{-m s_0
    {\rm dist}(k, \{-K,K\})}$, and hence we obtain $w_k' \leq C w_k$,
  which gives
  \begin{equation}
    \label{eq:consest_cb}
    \sum_{j \notin \{-K+1, \dots, K\}} \hspace{-4mm} \veps \|
    \sigmaa_{\by}-\sigma^{\mathrm{qc}}_{\by} \|_{L^\infty(Q_j)}^2 
    \leq C \veps^3 \sum_{k = -N}^N w_k |y_k''|^2.
  \end{equation}

  To compute the consistency error of the weak form in the atomistic
  region, we need to bound the difference
  $\bigl\lVert\sigmaa_{\by}-\sigma^{\mathrm{qc}}_{\by}\bigr\rVert_{\mrm{L}
    ^\infty(Q_j)}=\bigl\lVert\sigmaa_{\by}-\sigma^{\mrma}_{\by,*}\bigr\rVert_{\mrm{L}
    ^\infty(Q_j)}$ for $Q_j \subset \Omega^{\mrma}$. Using the same
  arguments as in the proof of Lemma \ref{lemma:contstressconsistency}
  we obtain
  \begin{displaymath}
    \bigl\lVert\sigmaa_{\by}-\sigma^{\mrma}_{\by,*}
    \bigr\rVert_{{L}^\infty(Q_j)}
    \leq C \big( \| \phi - \phi_{\mrma}^* \|_{L^\infty(Q_j)} + \veps
    \| \nabla\phi - \nabla\phi_{\mrma}^* \|_{L^\infty(Q_j)} \big).
  \end{displaymath}
  Lemma \ref{lemma:maxprinc} implies
  \begin{displaymath}
    \bigl\lVert\sigmaa_{\by}-\sigma^{\mrma}_{\by,*}
    \bigr\rVert_{{L}^\infty(Q_j)} \leq C \big( |\phi(a_L) -
    g_L^*(\by)| + |\phi(a_R) - g_R^*(\by)|\big) \e^{-\frac{m}{\veps}
      \min_{x \in Q_j} d_a(x)}.
  \end{displaymath}
  
  Next, we recall from Remark \ref{remark:refl} that $g_{R}^* =
  \gamma_{R}^* + \mathcal{O}(\tau)$, which is given by (cf. Remark
  \ref{remark:refl})
  \begin{equation*}
    \gamma_{R}(\by) = \int_\R \rho^{\text{refl}}_{\by}(x) G_\veps(\aR-x)\dx
    +\mc{O}(\tau),
  \end{equation*} 
  where $\rho^{\text{refl}}_{\by}(z) = \sum_{j \in \Z} \delta_\veps(z
  - y_j^{\rm refl})$ and $\by^{\rm refl}$ is a reflected and
  periodized extension of $(y_j)_{j = -K}^K$. Hence, we obtain
  \begin{align*}
     \big|\phi(\aR)-g_R^*(\by)\big| =~& \big|\phi(\aR) - \gamma_L(\by)\big| +\mc{O}(\tau)\\
      \leq~& \frac{1}{2m\veps}\Big|\int_\R \big(\rho_\by(z) -
      \rho_{\by}^{{\rm refl}}(z)\big) \onept\e^{-\frac{m}{\veps} |\aR-z|}\dz\Big| + \mc{O}(\tau).
 \end{align*}

 Minor modifications of the proofs of Lemma
 \ref{lemma:psijconvergence} and Lemma \ref{lemma:ypp} yield
 \begin{align}
   \notag
   \big|\phi(\aR)-g_R^*(\by)\big| \leq~& \frac{C}{\veps}\sum_{j = K+1}^\infty \big|
   y_j^{\rm refl} - y_j \big| \e^{-\frac{m}{\veps} (\min(y_j^{\rm
       refl}, y_j) - a_R)} + \mathcal{O(\tau)}\\
   \label{eq:consqc_errbc}
   \leq~& C \veps \sum_{n = 1}^\infty \| \by'' \|_{\ell^1([K-n+1,
     K+n])} n \e^{-mn s_0} + \mathcal{O}(\tau).
 \end{align}
 An analogous result holds for $|\phi(\aL)-g_L^*(\by)|$. It is now
 straightforward to see that the consistency error committed in the
 atomistic region can be bounded above in the same way as the
 consistency error committed in the continuum region (in fact it is
 dominated by \eqref{eq:consest_cb}. This completes the proof.
\end{proof}

\subsection{Stability}
\label{sec:m0:stab}
The special choice $g^*(\by)$ of boundary conditions for the atomistic subproblem allows for an
elementary stability analysis of $\Eqc(\by)$ that draws from the ideas we used in Section
\ref{sec:specialcase}. We recall that 
\begin{equation*}
\begin{split}
 \E^{\mrma}_*(\by) =~&\frac{1}{4m\veps}\int_{\Omega^{\mrma}}\!\int_{\Omega^{\mrma}}\rho_{\by}(x)
\bigl(\e^{-\tfrac{m}{\veps}|x-z|}
+\e^{-\tfrac{m}{\veps}(2\aR(\by)-x-z)}\\
&\hspace{5.0cm}+\e^{-\tfrac{m}{\veps}(x+z-2\aL(\by))}\bigr)\rho_{\by}
(z)\dz\dx+\mc { O } (\tau).
\end{split}
\end{equation*}
The next result addresses the differentiability of $\gamma_L$ and $\gamma_R$. We show that the
derivatives satisfy certain bounds. 

\begin{lemma}
\label{lemma:gammaLR}
Let $\by\in\Omega_{\aLR}^{2K+1}$ satisfy $y_{i+1}-y_i>\veps\vsig$ for all $i\in\{-K+1,\ldots,K\}$,
$\aR-y_K>\veps\vsig/2$, and $y_{-K}-\aL>\veps\vsig/2$. Then,
 $\gamma_L(\by)$ is twice continuously differentiable with respect to $\by$ and $\aLR$ and
there exists $C(m\min\by')$ (independent of $\veps$) such that
\begin{equation*}
\begin{split} 
\bigl|D\gamma_L(\by,\aLR) \Didot (\bu,h)\bigr| \leq~& C(m\min\by')\biggl(
\biggl(\frac{u_{-K}-h_L}{\veps}\biggr)^{\!\!2} + \sum_{k=-K+1}^K (u_k')^2
\biggr)^{\!1/2},
\\
 \bigl|D^2\gamma_L(\by,\aLR) \Didot \bigl[(\bu,h),(\bu,h)\bigr] \bigr|\leq~& 
C(m\min\by')\biggl( \biggl(\frac{u_{-K}-h_L}{\veps}\biggr)^{\!\!2}+\sum_{k=-K+1}^K (u_k')^2 \biggr)
\end{split}
\end{equation*}
for all $\bu\in\mc{U}$ and $h\in\R^2$. Analogous bounds hold for $\gamma_R(\by,\aLR)$.
\end{lemma}
\begin{proof}
  The proof is based on the observation that
  \begin{equation*}
    \begin{split}
      \gamma_L(\by,\aLR) = \frac{1}{m}\sum_{j=-K}^K\int_{\Omega_{\aLR}}
      \e^{-\tfrac{m}{\veps}(x-\aL)}\delta_\veps(x-y_j)\dx = \frac{\ceff }{m}
      \e^{-\tfrac{m}{\veps}(y_{-K}-\aL)}\sum_{j=-K}^K \e^{-\tfrac{m}{\veps}(y_j-y_{-K})}.
    \end{split}
  \end{equation*}
  The rest of the proof is a straightforward computation; see
  \cite[Lemma 5.3]{BLthesis} for the details.
\end{proof}

The $\tau$-dependent terms in
$\E^{\mrma}_*(\by)=\E_{\aLR(\by),g^*(\by)}(\bya)$ from
\eqref{eq:Eagwithtau} only contain $\gamma_L(\by)$ and
$\gamma_R(\by)$, whose derivatives are bounded by Lemma
\ref{lemma:gammaLR}. The derivatives of these $\tau$-dependent terms
are therefore still of order $\mc{O}(\tau)$ and will be neglected in
the proof of the following result.

\begin{lemma}
Let $\by\in\mc{Y}$ satisfy $\min\by'>\vsig$. Then,
\begin{equation*}
D^2\E^{\qc}(\by)\Didot[\bu,\bu]\geq \Bigl(\frac{m\ceff^2}{2}\onept
\e^{-m\max\by'} -\mc{O}(\tau)\Bigr) \norm{\bu'}^2_{\ell^2_\veps}\quad \forall \bu\in\mc{U}.
\end{equation*}
\label{lemma:QCStabMethod0}
\end{lemma}

\begin{proof}
We treat continuum and atomistic contributions independently and start with the former. Lemma
\ref{lemma:cbstability} states that
\begin{equation*}
 D^2\E^{\cb}_j(\by)\Didot[\bu,\bu] \geq \frac{m^2\ceff^2}{2}\onept\e^{- m\max \by'}\veps
|u_j'|^2
\end{equation*}
for all $j=-N,\ldots,N$. Hence, the definition \eqref{eq:Ecbdefinition} of $\E_*^{\cb}$
directly implies that
\begin{equation*}
 D^2\E_*^{\cb}(\by)\Didot[\bu,\bu]\geq \e^{- m\max
\by'}\frac{m^2\ceff^2}{2}\onept\veps\biggl(\,\sum_{j=-N}^{-K-1}|u_j'|^2 +
\frac{1}{2}\bigl(|u'_{-K}|^2+|u_{K+1}'|^2\bigr)+\sum_{j=K+2}^{N}|u_j'|^2\biggr).
\end{equation*}
Let us now turn to the atomistic part $\E^{\mrma}_*(\by)$. From Section \ref{sec:specialcase} we
know that for the given choice of boundary conditions and $a(\by)$ we can write the energy of the
atomistic part as
\begin{align}
 \E^{\mrma}_*(\by) =~& \frac{\veps}{4m}\sum_{i,j=-K}^K\int_{\Omega^{\mrma}}\int_{\Omega^{\mrma}}
\delta_\veps(x-y_i)\bigl(\e^{-\tfrac{m}{\veps}|x-z|} +
\e^{-\tfrac{m}{\veps}(x+z-y_{-K-1}-y_{-K})}\nonumber\\[-3mm]
&\hspace{6.1cm}+\e^{-\tfrac{m}{\veps}(y_{K+1}+y_K - x-z)}\bigr)\delta_\veps(z-y_j)\dz\dx \nonumber
\\ =~&\frac{\veps \ceff^2}{4m}\sum_{i,j=-K}^K \bigl( \e^{-\tfrac{m}{\veps}|y_i-y_j|} +
\e^{-\tfrac{m}{\veps}(y_i+y_j-y_{-K}-y_{-K-1})} \label{eq:E*atlong} \\[-3mm]
&\hspace{3.7cm}+\e^{-\tfrac{m}{\veps}(y_K+y_{K+1} -
  y_i-y_j)}\bigr)+\E_{{\rm self}} +
\mc{O}(\tau),\nonumber
\end{align}
where the constant $\E_{{\rm self}}$ accounts for the self-energies of
the atoms $\{-K,\ldots,K\}$.  Differentiating twice and keeping only
contributions from nearest neighbour interactions leads directly to
\begin{equation*}
 D^2\E^{\mrma}_*(\by)\Didot[\bu,\bu] \geq \e^{-m \max \by'}\frac{m\ceff^2}{2}\,
\veps\biggl(\frac{1}{2}|u_{-K}'|^2 + \sum_{i=-K+1}^{K}|u_i'|^2 +
\frac{1}{2}|u_{K+1}'|^2\biggr)-\mc{O}(\tau).
\end{equation*}
Adding the lower bounds for $D^2\E_*^{\cb}(\by)\Didot[\bu,\bu]$ and
$D^2\E^{\mrma}_*(\by)\Didot[\bu,\bu]$ we arrive at
\begin{equation*}
\begin{split}
  D^2\E^{\qc}(\by)\Didot[\bu,\bu] =~& 
\bigl(D^2\E_*^{\cb}(\by)+D^2\E^{\mrma}_*(\by)\bigr)\Didot[\bu,\bu]
\geq \Bigl(\e^{-m \max
\by'}\frac{m\ceff^2}{2}-\mc{O}(\tau)\Bigr) \,\bigl\lVert \bu'\bigr\rVert^2_{\ell^2_\veps}, 
\end{split}
\end{equation*}
for all $\bu\in\mc{U}$, as desired.
\end{proof}

\subsection{Error Estimates}
Combining the consistency and stability results we obtain the
following error estimates. We note that the upper bound on the error
depends on the smoothness of $\bar\by$ in the continuum region, but
that the dependence on $\bar\by$ in the atomistic region decays
exponentially with distance to the a/c interface. In realistic
higher-dimensional models such an estimate would make it possible to
allow defects in the atomistic region without affecting the error
estimate.


\begin{theorem}
  \label{Theorem:QC1Convergence}
  Suppose that $\bar\by\in\arg\min E_{\bs{f}}$ and $\bar\by_\qc \in \arg\min
  E_{\bf{f}}^\qc$ satisfy
  \begin{equation}
    \label{eq:err:bounds_yp}
    \min \bar\by', \min \bar\by'_\qc \geq s_0 \geq \varsigma_0, \quad \text{and}
    \quad
    \max \bar\by', \max \bar\by'_\qc \leq S_0 < +\infty.
  \end{equation}
  There exist constants $c$ and $C = C(s_0, S_0)$ such that, if
  $\Delta a \geq c \log(S_0)$, then
  \begin{equation}
    \label{eq:err:main_est}
    \bigl\| \bar{\by}'-\bar{\by}'_{\qc} \bigr\|_{\ell^2_\veps} \leq
    C \Big(\veps \big\| \bar\by'' \|_{\ell^2_{w, s_0}} +
    \tau\Big).
  \end{equation}
\end{theorem}
\begin{proof}
  From Lemma \ref{lemma:QCStabMethod0} it is clear that there exists a
  constant $c$ such that, for $\Delta a \geq c$, we have
  \begin{displaymath}
    D^2\E^{\qc}(\by)\Didot[\bu,\bu]\geq \frac{m\ceff^2}{4}\onept
    \e^{-m S_0} \norm{\bu'}^2_{\ell^2_\veps}\quad \forall
    \bu\in\mc{U},
    \quad \forall \by \in \mc{Y}, \by' \leq S_0.
  \end{displaymath}
  In particular, this holds for all $\by \in {\rm conv}\{\bar\by,
  \bar\by_\qc\}$. Let $c_0 = \frac{m\ceff^2}{4}\onept \e^{-m S_0}$.

  Let $\bu = \bar\by-\bar\by_\qc$; then we can choose $\by \in {\rm
    conv}\{\bar\by, \bar\by_\qc\}$ such that
  \begin{displaymath}
    c_0 \norm{\bu'}^2_{\ell^2_\veps} \leq
    D^2\E^{\qc}(\by)\Didot[\bu,\bu]
    = \big(D\E^\qc(\bar\by) - D\E^\qc(\bar\by_\qc)\big)[\bu].
  \end{displaymath}
  Employing the consistency estimate of Theorem
  \ref{Lemma:Method0Consistency} we obtain the stated result.
\end{proof}

\begin{remark}
  With some additional work it is possible to avoid assuming the
  existence of $\bar\by_\qc$, but deduce it from an inverse function
  theorem type argument \cite{OrtnerQNL, BLthesis}.
\end{remark}

\section{Boundary Conditions From Cell Problems}
\label{sec:method_dir}
\begin{figure}
\begin{center}
\includegraphics[width=.8\linewidth]{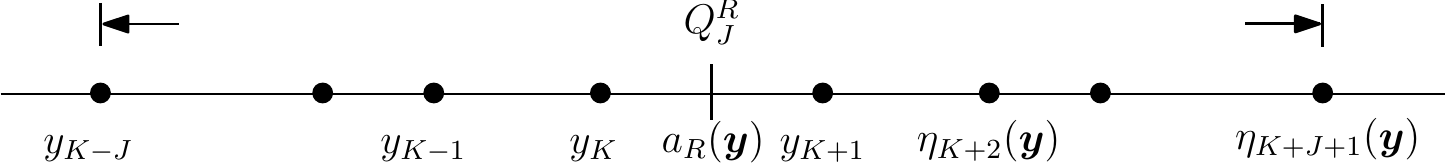}
\caption{Illustration of the problem in the interval $Q_J^R=(y_{K-J},2\aR(\by)-y_{K-J})$ used to
compute $g_R(\by)$.}
\label{fig:QC2}
\end{center}
\end{figure}

The boundary conditions $g^*(\by)$ we imposed on the atomistic
subproblem in Section \ref{sec:firstmethod} gave rise to a method
without ghost forces, and whose analysis was relatively
straightforward. The reasons for this is the clean weak formulation
\eqref{eq:QCweakform1} of $D\E^{\qc}$ and the convenient stability
properties established in Lemma \ref{lemma:QCStabMethod0}. We now
investigate how this situation changes if computationally cheaper
boundary conditions are chosen. The following construction may also
provide a starting point for generalisations to higher dimensions. 

For example, a canonical choice, which requires no additional
computational effort, is
\begin{equation}
  \label{eq:new_gLR}
  g_L(\by) = \psi^{(-K)}(\aL) \qquad 
  \text{and} \qquad g_R(\by) = \psi^{(K+1)}(\aR),
\end{equation}
where we still assume $\aL(\by)=\frac{1}{2}(y_{-K-1}+y_{-K})$ and
$\aR(\by) = \frac{1}{2}(y_K+y_{K+1})$. In this case, we have the
following result, which suggests that the additional error committed
can be controlled.

\begin{lemma}
  \label{th:dir:gLR_error}
  Let $\min\by \geq s_0 \geq \varsigma_0$ and let $g_{L/R}$ be given
  by \eqref{eq:new_gLR}; then,
  \begin{equation}
    \label{eq:dir:err_gLR}
    \big| g(\by) - g^*(\by) \big| \leq C\big( \veps^{1/2} \| \by'' 
    \|_{\ell^2_{w, s_0}} + \tau \big). 
  \end{equation}
\end{lemma}
\begin{proof}
  Without loss of generality we focus on $g_R$ only. Upon first
  estimating
  \begin{displaymath}
    \big| g_{R}(\by) - g_{R}^*(\by) \big| \leq \big| \psi^{(K+1)}(a_R)
    - \phi(a_R) \big| + \big| \phi(a_R) - g_R^*(\by) \big|,
  \end{displaymath}
  and then employing \eqref{eq:consqc_errbc} and Lemma
  \ref{lemma:psijconvergence}, we obtain
  \begin{displaymath}
    \big| g_{R}(\by) - g_{R}^*(\by) \big| \leq C \veps \sum_{n =
      1}^\infty \| \by'' \|_{\ell^1([K-n+1, K+n])} n \e^{-mn s_0}
    + \mathcal{O}(\tau).
  \end{displaymath}
  Using the same argument as in \eqref{eq:cons_err_contregion} to
  \eqref{eq:consest_cb}, we obtain the upper bound
  \eqref{eq:dir:err_gLR}.
\end{proof}

Motivated by Lemma \ref{th:dir:gLR_error} we define a second
a/c energy $\Eqc(\by)$ by
\begin{equation}
 \Eqc(\by) = \Ec_*(\by) + \Ea(\by),
\label{eq:EnergyMethod2}
\end{equation}
where $\Ec_*(\by)$ is the same as in the method discussed in Section \ref{sec:firstmethod} (see
\eqref{eq:Ecbdefinition}) and
\begin{equation*}
\begin{split}
\E^{\mrma}(\by) =~& \E_{\aLR(\by),g(\by)}(\bya)\\
=~&-\inf \Bigl\{ I_{\aLR(\by)}(\varphi,\bya):\ \ \varphi\in\mrm{H}^1(\Omega^{\mrma}),\quad
\varphi|_{\partial\Omega^{\mrma}}=g(\by) \Bigr\}.
\end{split}
\end{equation*}
We denote the minimizer for given $\by$ by
$\phi_{\mrma}\in\mrm{H}^1(\Omega^{\mrma})$.

Before we embark on the analysis of this new method, we establish a
useful auxiliary result.

\begin{lemma}
  \label{th:dir:bounds_DgRL}
  Let $\min\by' \geq s_0 \geq \varsigma_0$. Let $g_{R}(\by)$ be
  defined by \eqref{eq:new_gLR}; then, it can be equivalently written as
  \begin{equation}
    \label{eq:dir:gR_equiv}
     g_R(\by) = \frac{1}{2m}\sum_{k \in \Z} \int_\R \delta_\veps\big(z -
     (k+\smfrac12) \veps y_{K+1}'\big) \e^{-\frac{m}{\veps} |z|} \dz.
  \end{equation}
  In particular, $g_R$ is twice Fr\'{e}chet differentiable with
  respect to $\by$, and there exists a constant $C = C(s_0)$ such
  that, for all $\by \in \mc{Y}$ with $\min\by' \geq s_0 \geq
  \varsigma_0$,
  \begin{displaymath}
    \big|D_\by g_{R}(\by) \cdot \bu \big| \leq C |u_{K+1}'| \quad
    \text{and} \quad
    \big|D_\by^2 g_R(\by) \cdot [\bu, \bu] \big| \leq C |u_{K+1}'|^2 
    \qquad
    \forall \bu \in \mc{U}.
  \end{displaymath}
  Analogous results hold for $g_L(\by)$.
\end{lemma}
\begin{proof}
  Recall from \eqref{eq:phiandpsij} that
  \begin{displaymath}
    g_R(\by) = \psi^{(K+1)}(a_R) = \frac{1}{2m}\sum_{k \in \Z} \int_\R \delta_\veps(z-y_{K+k}^{(K+1)})
\e^{-\tfrac{m}{\veps}|a_R - z|}\dz,
  \end{displaymath}
  where $y_j^{(K+1)}$ denotes the periodic extension defined in
  \eqref{eq:cb:periodic_ext}. We use the identities
  \begin{displaymath}
    y_{K+k}^{(K+1)} = y_K + k \veps y_{K+1}' \quad \text{and} \quad
    a_R = y_K + \smfrac12 \veps y_{K+1}'
  \end{displaymath}
  to obtain
  \begin{displaymath}
    g_R(\by) = \frac{1}{2m}\sum_{k \in \Z} \int_\R \delta_\veps(z- y_K
    - k \veps y_{K+1}') \e^{-\smfrac{m}{\veps}|y_K + \frac12 \veps y_{K+1}' -z|}\dz.
  \end{displaymath}
  Shifting the integration by $(y_K + \smfrac12 \veps y_{K+1}')$, we
  obtain \eqref{eq:dir:gR_equiv}.

  The bound on the first and second derivatives follows as in Lemma
  \ref{lemma:gammaLR}; the key observation being that $\delta_\veps =
  \mathcal{O}(\veps^{-1})$ is balanced against the $\veps$ preceding
  $y_{K+1}'$ in its argument.
\end{proof}

\subsection{Consistency}
A crucial difference between the a/c energy \eqref{eq:EnergyMethod2}
and the energy from Section \ref{sec:firstmethod} is that now the
derivative of the atomistic energy with respect to the boundary
conditions does not vanish. 


  Since the continuum contribution to $D\E^{\qc}(\by)\Didot\bu$ is the
  same as in Section \ref{sec:firstmethod} we only need to analyze
  $D\E^{\mrma}(\by)$. Using the chain rule we obtain
  \begin{equation*}
    \begin{split}
      D\E^{\mrma}(\by)\Didot\bu =~&  D_{\bya}\E_{\aLR(\by),g(\by)}(\bya)\Didot \bu_{\mrma} +
      D_{\aLR}\E_{\aLR(\by),g(\by)}(\bya)\Didot\aLR(\bu)\\ 
      &~ + D_{g}\E_{\aLR(\by),g(\by)}(\bya)\Didot (D_{\by}g(\by)\Didot \bu).
    \end{split}
  \end{equation*}
  The same reasoning as in Section \ref{sec:firstmethod} gives for the
  first two terms on the right-hand side
  \begin{equation}
    D_{\bya}\E_{\aLR(\by),g(\by)}(\bya)\Didot \bu_{\mrma} +
    D_{\aLR}\E_{\aLR(\by),g(\by)}(\bya)\Didot\aLR(\bu) = \int_{\Omega^{\mrma}} \sigma^{\mrma}_{\by}(x)
    \nabla u(x)\dx,
    \label{eq:firstTermQC2}
  \end{equation}
  where $\sigma_{\by}^{\mrma}(x)$ is given by \eqref{eq:sigmaat12} with $\phi=\phi_\mrma$.

  Next, we turn our attention to the term
  $D_{g}\E_{\aLR(\by),g(\by)}(\bya)\cdot (D_{\by}g(\by)\cdot \bu)$. We
  recall from Lemma~\ref{lemma:depbdrycond} that (for $\Delta a \gg
  \veps$)
  \begin{equation*} 
    D_{g}\E_{a(\by),g(\by)}(\bya) = -m\veps\bigl[g_L(\by)-g_{L}^*(\by),\ \
    g_R(\by)-g_{R}^*(\by)\bigr]+\mc{O}(\veps\tau).
  \end{equation*}
  Combining this result with Lemma \ref{th:dir:bounds_DgRL} and Lemma
  \ref{th:dir:gLR_error}, we obtain
  \begin{align}
    \notag
    \big| D_{g}\E_{\aLR(\by),g(\by)}(\bya)\cdot (D_{\by}g(\by)\cdot
    \bu) \big| \leq~& C \veps\Big(|g_L - g_L^*| + |g_R -
    g_R^*| + \tau \Big) \big(|u_{-K}'|^2 + |u_{K+1}'|^2\big)^{1/2} \\
    \label{eq:dir:cons:second_term}
    \leq~& C \Big( \veps \|\by''\|_{\ell^2_{w,s_0}} +
      \tau \Big) \cdot \|\bu'\|_{\ell^2_\veps},
  \end{align}
  where $C = C(\min\by')$, and we have estimated $\veps^{1/2}\tau \leq
  \tau$. Equipped with these estimates, we obtain the following
  consistency result.

\begin{lemma}
  \label{th:dir:consistency}
  Let $\by \in \mc{Y}$ with $\min\by' \geq s_0 \geq \varsigma_0$; then,
  there exists a constant $C = C(s_0)$ such that
  \begin{displaymath}
    \Big| D\E(\by)\Didot \bu  - D\E^{\mrma}(\by)\Didot\bu \Big|
    \leq C \Big( \veps \|\by''\|_{\ell^2_{w,s_0}} + \tau \Big)
    \,\|\nabla u \|_{L^2} \qquad \forall \bu \in \mc{U},
  \end{displaymath}
  where we have used the same notation as in Theorem
  \ref{Lemma:Method0Consistency}.
\end{lemma}
\begin{proof}
  From \eqref{eq:firstTermQC2} and \eqref{eq:dir:cons:second_term} we
  obtain that
  \begin{align*}
    \Big| D\E(\by)\Didot \bu  - D\E^{\mrma}(\by)\Didot\bu \Big| \leq~&
    \bigg[\Big( \veps \sum_{\substack{j = -N, \dots, N \\ j \notin
        \{-K+1,\dots,K\}}} \big\| \sigma_\by - \sigma_{j, \by}^{\rm
      cb} \big\|_{L^\infty(Q_j)}^2 + \big\| \sigma_\by - \sigma_\by^{\rm
      at} \big\|_{L^2(a_L, a_R)}^2\Big)^{1/2} \\
    & \hspace{3cm} + C \big( \veps \|\by''\|_{\ell^2_{w,s_0}} +
      \tau \big) \bigg] \cdot \| \nabla u \|_{L^2}.
  \end{align*}
  The first group in the upper bound was already estimated in the
  proof of Theorem \ref{Lemma:Method0Consistency}, and the second
  group, $\| \sigma_\by - \sigma_\by^{\rm at} \|_{L^2(a_L, a_R)}$ can
  be treated analogously to the term $\|\sigma_\by -
  \sigma_{\by,*}^{\rm at} \|_{L^2(a_L, a_R)}$ in the proof of Theorem
  \ref{Lemma:Method0Consistency}.
\end{proof}

\subsection{Stability}
We wish to compute a convenient lower bound on
$D^2\Eqc(\by)\Didot[\bu,\bu]$ for some given $\by \in \mc{Y}$ with
$\vsig \leq s_0 \leq \by' \leq S_0$.  Since the continuum part of the
energy is the same as in the first method, we only address the
stability of the atomistic subproblem with the given choice of
boundary data. We write the second derivative of the energy
$\E^{\mrma}$ in the form
\begin{equation*}
 D^2\Ea(\by)\Didot[\bu,\bu] =
D^2\E^{\mrma}_*(\by)\Didot[\bu,\bu]+\bigl(D^2\Ea(\by)-D^2\E^{\mrma}_*(\by)\bigr)\Didot[\bu,\bu]
\end{equation*}
and use the coercivity of $D^2\E^{\mrma}_*(\by)$: we know from Lemma \ref{lemma:QCStabMethod0} that
\begin{equation*}
 D^2\E^{\mrma}_*(\by)\Didot[\bu,\bu] \geq \e^{-m \max \by'}\frac{m\ceff^2}{2}\,
\veps\biggl(\frac{1}{2}|u_{-K}'|^2 + \sum_{i=-K+1}^{K}|u_i'|^2 +
\frac{1}{2}|u_{K+1}'|^2\biggr)-\mc{O}(\tau)
\end{equation*}
for all $\bu\in\mc{U}$; hence we are left to analyze the difference
$D^2\Ea(\by)-D^2\E^{\mrma}_{*}(\by)$. We will {\em not} show that this
difference is small, but will only be able to bound it below by a
controllable quantity. This is reminiscent of similar observations
made in \cite{OrtnerQNL}.


\begin{lemma}
  \label{th:dir:hess_diff}
  Let $\by \in \mc{Y}$ such that $\min\by' \geq s_0 \geq \varsigma_0$; then
  there exists a constant $C = C(s_0)$ such that
  \begin{displaymath}
    \bigl(D^2\Ea(\by)-D^2\E^{\mrma}_*(\by)\bigr)\Didot[\bu,\bu]
    \geq - C \big( \veps^{1/2} \|\by''\|_{\ell^2_{w, s_0}} + \tau\big).
  \end{displaymath}
\end{lemma}
\begin{proof}
  The difference between the energies $\E^{\mrma}(\by)$ and
  $\E^{\mrma}_*(\by)$ only consists of effects from the boundary
  conditions. We have, by \eqref{eq:Isimpleexpression},
  \begin{equation*}
    \begin{split}
      \Ea(\by)-\E^{\mrma}_*(\by) =~& -I_{\aLR(\by)}(\xi_{\aLR(\by),g(\by)},\by) +
      I_{\aLR(\by)}(\xi_{\aLR(\by),g^*(\by)},\by)
      =\frac{m\veps}{2}\bigl|g(\by)-g^*(\by)\bigr|^2 +\mc{O}(\veps\tau).
    \end{split}
  \end{equation*}
  As in Section \ref{sec:m0:stab}, one can verify that the
  $\mathcal{O}(\tau)$ term remains of that same order in the first and
  second derivatives.  This implies that
  \begin{align}
    \notag
    \bigl(D^2\Ea(\by)-D^2\E^{\mrma}_*(\by)\bigr)\Didot[\bu,\bu] =~& m\veps
    \bigl(g(\by)-g^*(\by)\bigr)^{\rm T}\bigl[\bigl(D^2g(\by)-D^2g^*(\by)\bigr)\Didot[\bu,\bu]\bigr]\\
    \notag
    & \hspace{1cm} + 2m\veps \bigl|\bigl(Dg(\by)-Dg^*(\by)\bigr)\Didot\bu\bigr|^2
    +\mc{O}(\tau \|\bu'\|_{\ell^2_\veps}^2) \\
    & \hspace{-2.5cm} \geq m\veps
    \bigl(g(\by)-g^*(\by)\bigr)^{\rm T}\bigl[\bigl(D^2g(\by)-D^2g^*(\by)\bigr)\Didot[\bu,\bu]\bigr]
    + \mc{O}(\tau \|\bu'\|_{\ell^2_\veps}^2).
    \label{eq:D2EmD2E}
  \end{align}
  We now employ Lemma \ref{th:dir:bounds_DgRL} to bound $D^2 g(\by)$,
  Lemma \ref{lemma:gammaLR} to bound $D^2 g^*$ (up to another
  $\mc{O}(\tau)$ error), and Lemma \ref{th:dir:gLR_error} to bound $g
  - g^*$, which yields
  \begin{displaymath}
    \bigl(D^2\Ea(\by)-D^2\E^{\mrma}_*(\by)\bigr)\Didot[\bu,\bu]
    \geq - C \big( \veps^{1/2} \|\by''\|_{\ell^2_{w, s_0}} + \tau\big) \|\bu'\|_{\ell^2_\veps}^2,
  \end{displaymath}
  where $C = C(\min\by')$.
\end{proof}

From Lemma \ref{th:dir:hess_diff} and Lemma
\ref{lemma:QCStabMethod0} we immediately obtain the following
corollary, which states that, if $S_0$ is moderate, $\by$ ``smooth''
in a neighbourhood of the interfaces $a_{L/R}$ and in the continuum region,
and if the atomistic region is sufficiently large, then
$D^2\E^{\qc}(\by)$ is stable.

\begin{corollary}
  \label{th:dir:stab}
  Let $\by \in \mc{Y}$ satisfy $\min\by' \geq s_0 \geq \varsigma_0$
  and $\max \by' \leq S_0$; then there exists a constant $C = C(s_0)$
  such that
  \begin{equation*}
    D^2\E^{\qc}(\by)\Didot[\bu,\bu]\geq \Bigl(\frac{m\ceff^2}{2}\onept
    \e^{-m S_0} - C \big( \veps^{1/2} \|\by''\|_{\ell^2_{w, s_0}} +
    \tau\big) \Bigr) \norm{\bu'}^2_{\ell^2_\veps}\qquad \forall \bu\in\mc{U}.
  \end{equation*}
\end{corollary}

\begin{remark}
  The scaling $\veps^{1/2}$ is due to the fact that the additional
  error committed is concentrated in a region of length $\veps$.
\end{remark}

\subsection{Error Estimates}
Repeating the proof of Theorem \ref{Theorem:QC1Convergence} verbatim,
but replacing the consistency and stability estimates from Section
\ref{Sec:QCCoupling} with those derived in Lemma
\ref{th:dir:consistency} and Corollary \ref{th:dir:stab}, we obtain
the following error estimates for the modified a/c method.

\begin{theorem}
  \label{Theorem:dir_convergence}
  Recall the notation introduced in Theorem
  \ref{Lemma:Method0Consistency}.  Suppose that $\bar\by\in\arg\min
  E_{\bs{f}}$ and $\bar\by_\qc \in \arg\min E_{\bf{f}}^\qc$, where
  $\E^\qc$ is defined in \eqref{eq:EnergyMethod2}, satisfy
  \begin{equation}
    \label{eq:err:bounds_yp}
    \min \bar\by', \min \bar\by'_\qc \geq s_0 \geq \varsigma_0, \quad \text{and}
    \quad
    \max\bar\by', \max\bar\by'_\qc \leq S_0 < +\infty.
  \end{equation}
  There exist constants $c$ and $C = C(s_0, S_0)$ such that, if $\tau
  + \veps^{1/2} \|\by''\|_{\ell^2_{w,s_0}} \leq c \e^{-m S_0}$ (in
  particular, $K$ must be sufficiently large), then
  \begin{equation}
    \label{eq:err:main_est}
    \bigl\| \bar{\by}'-\bar{\by}'_{\qc} \bigr\|_{\ell^2_\veps} \leq
    C \Big( \veps \big\| \bar\by'' \|_{\ell^2_{w, s_0}} + \tau \Big).
  \end{equation}
\end{theorem}

\section{Conclusions and Outlook}
\label{sec:SM_Outlook}
We have presented a rigorous error analysis of an
atomistic-to-continuum coupling method for a field-based interaction
potential in one space dimension. The starting point for the design of
coupling methods was a weak formulation of the forces arising from the
atomistic model. This provided a natural connection point to the
corresponding continuum model. We believe that the present work in a
comparably simple setting addresses several important questions
relevant for a/c coupling in the presence of fields, most prominently
the dependence of the a/c methods on choice of the boundary and the
boundary data for the interaction fields.

For the two a/c methods we discussed we chose $\by$-dependent
boundaries $\aLR(\by)$ of the atomistic subdomain $\Omega^{\mrma}$. In
other words we fixed the position of the boundary in the Lagrangian
domain. This leads to convenient weak formulations of
$D\E^{\qc}(\by)$. An obvious alternative (particularly relevant for
higher dimensions) is the choice of $\by$-independent $\aLR$. We have
not investigated this further, however, see \cite{BLthesis} for some
preliminary remarks.

We also remark that we heavily utilized the one-dimensional setting in
several places in the analysis. A generalisation both of the numerical
methods and their analysis is therefore non-trivial. In particular, we
can see no straightforward generalisation of the reflection boundary
conditions $g^*(\by)$. A possible way forward would be to give an
alternative analysis of the second method described in Section
\ref{sec:method_dir} that does not utilize these reflection
techniques.

\bibliographystyle{plain}
\bibliography{qcfield_biblio}

\end{document}